\documentclass{elsarticle}

\usepackage[utf8]{inputenc}
\usepackage[T1]{fontenc}
\usepackage{amsfonts,amsthm,amsmath,amssymb}
\usepackage{mathabx}
\usepackage{cleveref}
\usepackage{graphicx}
\usepackage{tikz}
\usepackage{enumitem}
\usepackage{float}
\usepackage{multirow}
\usepackage{todonotes}
\usepackage{url}

\setlist[enumerate]{noitemsep, nolistsep}
\setlist[itemize]{noitemsep, nolistsep}

\newtheorem{theorem}{Theorem}
\newtheorem{lemma}[theorem]{Lemma}
\newtheorem{corollary}[theorem]{Corollary}
\newtheorem{conjecture}{Conjecture}
\theoremstyle{definition}
\newtheorem{definition}{Definition}

\newcommand{\N}{\mathbb{N}}
\newcommand{\NP}{\mathcal{NP}}
\newcommand{\NPC}{\mathcal{NPC}}

\begin{document}

\begin{frontmatter}
    \title{Backbone coloring for graphs with degree $4$}
    
    \author{Krzysztof Turowski\fnref{kt}}
    \address{Theoretical Computer Science Department, Jagiellonian University, Kraków, 30-348 Poland}
    \fntext[kt]{E-mail: krzysztof.szymon.turowski@gmail.com. Corresponding author}
    
    \author{Krzysztof Michalik\fnref{km}}
    \address{Theoretical Computer Science Department, Jagiellonian University, Kraków, 30-348 Poland}
    \fntext[km]{E-mail: krzysztof.michalik@student.uj.edu.pl}

    \begin{abstract}
      The $\lambda$-backbone coloring of a graph $G$ with backbone $H$ is a proper vertex coloring in which the endpoints of every edge from $H$ are assigned colors which differ by at least $\lambda$.

      In this paper we prove that $BBC_\lambda(G,H)\le\lambda+5$ for every graph $G$ with maximum degree at most $4$ and every connected spanning bipartite backbone $H$, provided that $\lambda\ge3$. A suitable coloring can be found in linear time. The bound is sharp for $\lambda\ge4$, even for spanning tree backbones. The proof combines degree-list coloring with a balanced choice of vertices in an auxiliary bipartite multigraph.

      We also settle the previously open complexity status for graphs with maximum degree $4$ and connected bipartite backbones, proving $\NP$-completeness even for spanning tree backbones for every fixed $\lambda\ge2$. Furthermore, we obtain an almost complete classification of the complexity of deciding whether $BBC_{\lambda}(G,H)\le k$ for graphs with maximum degree at most $4$ and matching, galaxy, Hamiltonian path, and spanning tree backbones. We also determine the effect of allowing isolated vertices in matching and galaxy backbones.
    \end{abstract}
    \begin{keyword}
        graph coloring \sep backbone coloring \sep bounded-degree graphs
    \end{keyword}
\end{frontmatter}

\section{Introduction}

$\lambda$-backbone coloring is a problem in a vast family of approaches to graph coloring. It is motivated by a particular approach to the frequency assignment problem in radio networks, often represented by graphs with transmitters modelled as vertices.
There exists a way to model restrictions that a coloring imposes on the allowed level of interference between the transmitters: vertices in close proximity cannot have the same frequency. Additionally, the backbone requirements are understood as follows: there exists also a specific subset of pairs of vertices that require a wider separation bandwidth between the signals they are emitting, e.g. because they fulfill more critical roles in the network. Of course, our aim is to minimize the total bandwidth used by the whole radio network \cite{broersma2007backbone}.

We use standard graph terminology as in \cite{murty2008graph}. Graphs are finite, simple and undirected, unless stated otherwise. We write $n=|V(G)|$ and $m=|E(G)|$. We write $uv$ for the edge with endpoints $u$ and $v$, reserving braces for sets. A subgraph $H$ of $G$ is \emph{spanning} if $V(H) = V(G)$. For a vertex set $X$, $G \setminus X$ denotes vertex deletion; for a subgraph $H$, $G \setminus H$ denotes the graph $(V(G), E(G) \setminus E(H))$.

Formally, backbone coloring is defined as follows:
\begin{definition}[Broersma et al., \cite{broersma2007backbone}]
    Let $G$ be a graph and $H$ be a subgraph of $G$. Let also $\lambda \in \mathbb{N}_+$, $\lambda \ge 2$.
    The \emph{$\lambda$-backbone coloring} of $G$ with backbone $H$ is defined as a function $c\colon V(G) \to \mathbb{N}_+$ such that
    \begin{itemize}
        \item $c(u) \neq c(v)$ for every $uv \in E(G)$,
        \item $|c(u) - c(v)| \ge \lambda$ for every $uv \in E(H)$.
    \end{itemize}
\end{definition}

A $\lambda$-backbone $k$-coloring takes values in $\{1, \ldots, k\}$; not every color need be used. Adding the vertices of $V(G) \setminus V(H)$ to $H$ as isolated vertices does not change the coloring constraints. Thus, where convenient, we regard a backbone as spanning without restricting the definition.

The key difference between this problem and its ancestor -- classical vertex coloring -- is obviously an introduction of the backbone. This entails certain properties absent in the original problem: the latter is oblivious to the color ordering, and we can permute them in any order whenever necessary. However, in the backbone coloring problem the ordering of the colors matters, and it can be the case that even in the optimal coloring $c$ some colors between $1$ and $\max_{v \in V(G)} c(v)$ may not be used at all.
Thus, the \emph{$\lambda$-backbone chromatic number} for a graph $G$ with backbone $H$ (denoted as $BBC_\lambda(G, H)$) is defined accordingly as the smallest $k$ for which there exists a $\lambda$-backbone $k$-coloring of $G$ with backbone $H$.

In this paper we study backbone coloring for graphs with maximum degree at most $4$, focusing on connected bipartite backbones and the well-known backbone classes of matchings, galaxies (collections of vertex-disjoint stars), Hamiltonian paths, and spanning trees.

\subsection{Previous results}

The first results proved by Broersma et al. for the $\lambda$-backbone chromatic number were general bounds in terms of $\chi(G)$:
\begin{theorem}[\cite{broersma2007backbone}]
\label{thm:broersma}
Let $G$ be a graph and $H$ its spanning subgraph.
Then $\chi(G) \le BBC_{\lambda}(G, H) \le \lambda(\chi(G) - 1) + 1$.
\end{theorem}
In \cite{havet2014circular} and \cite{janczewski2015backbone} other general bounds were proposed, suited particularly for graphs with $\chi(H) \ll \chi(G)$:
\begin{theorem}[\cite{havet2014circular}]
\label{thm:havet}
Let $G$ be a graph and $H$ its spanning subgraph.
Then $BBC_{\lambda}(G, H) \le (\lambda + \chi(G) - 2) \chi(H) - \lambda + 2$.
\end{theorem}

\begin{theorem}[\cite{janczewski2015backbone}]
Let $G$ be a graph on $n$ vertices and $H$ its spanning subgraph. Then $\lambda(\chi(H) - 1) + 1 \le BBC_{\lambda}(G, H) \le \lambda(\chi(H) - 1) + n - \chi(H) + 1$.
\end{theorem}

For graphs with $\Delta(G)\le4$ and a connected spanning bipartite backbone $H$, \Cref{thm:havet} gives the bound $BBC_\lambda(G,H)\le\lambda+6$ whenever $\chi(G)\le4$. By Brooks' theorem \cite{lovasz1975three,baetz2014brooks}, this includes every such graph other than $K_5$. For $G=K_5$, a bipartition $A,B$ of $H$ can instead be colored injectively with $1,\ldots,|A|$ on $A$ and $\lambda+|A|,\ldots,\lambda+4$ on $B$. Thus the bound $\lambda+6$ holds in this case as well.

However, it turns out that the problem of computing the exact value of $BBC_{\lambda}(G, H)$ quickly becomes $\NP$-hard:
\begin{theorem}[\cite{contributions}, see also \cite{broersma2009lambda}]
    Let $\lambda \ge 2$. For matching and galaxy backbones, assume that every vertex of $G$ is incident with a backbone edge.
    \begin{itemize}
        \item[$(a)$] The following problem is polynomially solvable for any $l \le \lambda + 1$: given a graph $G$ and a galaxy backbone $S$, decide whether $BBC_{\lambda}(G, S) \le l$.
        \item[$(b)$] The following problem is $\NP$-complete for all $l \ge \lambda + 2$: given a graph $G$ and a matching backbone $M$, decide whether $BBC_{\lambda}(G, M) \le l$.
        \item[$(c)$] The following problem is polynomially solvable for any $l \le \lambda + 2$: given a graph $G$ and a spanning tree $T$, decide whether $BBC_{\lambda}(G, T) \le l$.
        \item[$(d)$] The following problem is $\NP$-complete for all $l \ge \lambda + 3$: given a graph $G$ and a Hamiltonian path $P$, decide whether $BBC_{\lambda}(G, P) \le l$.
    \end{itemize}
\end{theorem}
That is, when the backbone is connected, there is a case $BBC_{\lambda}(G, T) \le \lambda + 2$ which can be seemingly solved more easily. In particular, an investigation of the proofs of cases $(a)$ and $(c)$ shows that we construct a suitable coloring if the answer is \texttt{YES} in $O(n + m)$ time.

Related work also considers other restrictions on $G$. Havet et al. in \cite{havet2014circular} established upper bounds and hardness results for forest backbones in planar graphs, while Ara{\'u}jo et al. in \cite{araujo2022backbone} gave polynomial-time algorithms for backbone coloring and its circular variant when $G$ has bounded treewidth. For complete graphs, Michalik and Turowski in \cite{michalik2024cliques} obtained linear-time additive approximation algorithms for tree and forest backbones. More recently, Ara{\'u}jo, Nisse, and Picasarri-Arrieta in \cite{araujo2025chordal} proved that $BBC_2(G,H)\le\chi(G)+O(\sqrt{\chi(G)})$ when $G$ is chordal and $H$ is a spanning forest.

Let us proceed now with results for graphs with a fixed maximum degree $\Delta(G)$.
First, we note in passing that the proof of the above theorem in \cite{contributions} may be used to establish the hardness of problems $BBC_\lambda(G, M) \le \lambda + 2$ and $BBC_\lambda(G, M) \le \lambda + 3$ for graphs with $\Delta(G) \le 6$ -- but unfortunately this construction relies on the hardness of $l$-coloring which in turn requires $\Delta(G) \ge l + 1$, so it does not establish the hardness for all possible thresholds (at least up to the upper bounds, when problems become trivial).

Mi{\v{s}}kuf et al. in \cite{mivskuf2009backbone} proved important upper bounds when $\lambda = 2$ and the backbone is a $d$-degenerate graph, that is, a graph whose every induced subgraph contains a vertex of degree at most $d$:
\begin{theorem}[\cite{mivskuf2009backbone}] 
    \label{miskuf:d}
    Let $G$ be a graph and let $H$ be a $d$-degenerate subgraph of $G$. Then, $BBC_2(G, H) \le \Delta(G) + d + 1$.
\end{theorem}
This immediately implies $BBC_2(G, H) \le \Delta(G) + 2$ for galaxy, path, and tree backbones, as all such graphs are $1$-degenerate.

Moreover, the authors claimed to improve the result for matching backbones. Although the original proof turned out to be flawed, it was later corrected in \cite{araujo2022backbone}, thus establishing that:
\begin{theorem}[\cite{mivskuf2009backbone,araujo2022backbone}] 
    \label{miskuf:matching}
    Let $M$ be a matching in a graph $G$. Then it holds that $BBC_2(G, M) \le \max\{3, \Delta(G) + 1\}$.
\end{theorem}
The term $3$ accounts for graphs of maximum degree at most $1$: a backbone edge requires colors $1$ and $3$.

For $\Delta(G) \le 4$, the greedy construction for $1$-degenerate backbones in \cite{mivskuf2009backbone} runs in $O(n)$ time: a degeneracy ordering is found in linear time and scanned once for each of at most six colors.
The corrected matching construction in \cite{araujo2022backbone} also uses forks: nonadjacent vertices $x,y$ with a common neighbor such that $G \setminus \{x,y\}$ is connected. These are not separation pairs. There are at most $\sum_v \binom{\deg_G(v)}{2} = O(n)$ candidate pairs when $\Delta(G) \le 4$, and each can be tested by a connectivity search in $O(n)$ time. Together with the depth-first searches and block decompositions in that proof, this gives an $O(n^2)$ implementation; we use this bound here.

The investigations for bounded $\Delta(G)$ were pursued further by Janczewski and Turowski in \cite{janczewski2015computational}, where the authors proved that for $\Delta(G)\le3$ one can find an optimal $\lambda$-backbone coloring of $G$ with any connected spanning backbone $H$ in $O(n^2)$ time.
In particular, they proved that in the hardest case, when $G$ is $3$-colorable and $H$ is bipartite and connected, it always holds that $BBC_\lambda(G,H)\leq\lambda+3$.

Their contribution for establishing hardness was twofold. First, for $\Delta(G) \ge 4$ they showed $\NP$-completeness in the general setting:
\begin{theorem}[\cite{janczewski2015computational}]
    For any graph $G$ with a fixed $\Delta(G) \ge 4$ and any $\lambda\geq 2$ the problem of verifying if $BBC_\lambda(G, G)\leq 2\lambda+1$ is $\NP$-complete.
\end{theorem}
Furthermore, they also presented a result for $\Delta(G) \ge 5$ when the backbone is just a spanning tree of $G$:
\begin{theorem}[\cite{janczewski2015computational}]
    For any graph $G$ with a fixed $\Delta(G) \ge 5$ and any $\lambda\geq 3$ the problem of verifying if $BBC_\lambda(G, T)\leq \lambda + 3$, where $T$ is a spanning tree of $G$, is $\NP$-complete.
\end{theorem}
Interestingly, the original statement of the above theorem in \cite{janczewski2015computational} assumed $\lambda \ge 4$, but their actual proof allows for improving the result and covering also the case $\lambda = 3$ as well.

Janczewski and Turowski conclude by leaving open the complexity status for graphs with maximum degree $4$ and a connected bipartite backbone \cite[p.~236]{janczewski2015computational}. The hardness result for $H=G$ does not settle this case, since that backbone need not be bipartite, while their spanning tree construction requires maximum degree at least $5$.

\subsection{Our work}

Our main contribution is a constructive upper bound for connected spanning bipartite backbones in graphs with maximum degree at most $4$. The following theorem improves the bound $\lambda+6$ obtained from \Cref{thm:havet} to $\lambda+5$:
\begin{theorem}
    \label{thm:connected-lambda+5}
    Let $G$ be a graph with $\Delta(G)\le4$ and $H$ be its connected spanning bipartite subgraph. Then for every fixed $\lambda\ge3$ it holds that $BBC_\lambda(G,H)\le\lambda+5$. A suitable coloring can be found in $O(n)$ time.
\end{theorem}
The result applies to backbones containing cycles as well as to trees, and is sharp for $\lambda\ge4$ even for spanning tree backbones. We prove it in \Cref{sec:connected-backbone}. The proof uses the bipartition of the backbone to reduce the problem to coloring a graph of maximum degree at most $3$. We handle its exceptional $K_4$ components by a balanced choice of vertices and degree-list coloring.

Our hardness results answer the open question from \cite{janczewski2015computational}. By \Cref{thm:tree-lambda+3}, deciding whether $BBC_\lambda(G,T)\le\lambda+3$ is $\NP$-complete for graphs with $\Delta(G)\le4$ and a spanning tree $T$, for every fixed $\lambda\ge2$. This lowers the maximum degree needed for their spanning tree hardness result from $5$ to $4$ and also includes $\lambda=2$. Together with their polynomial-time algorithm for subcubic graphs, it settles the degree boundary for spanning tree and connected bipartite backbones.

We also establish hardness results and sharper upper bounds for well-known classes of backbones considered in the literature: matchings, galaxies, Hamiltonian paths, and spanning trees. These results are highlighted in \Cref{tab:summary}. We complete the decision complexity classification in the table for matching and galaxy backbones, and leave only two cases open within these four classes: $BBC_2(G,P)\le5$ and $BBC_3(G,T)\le7$.

In addition, we prove $\NP$-completeness at $\lambda+2$ colors for matching backbones and at $\lambda+3$ colors for galaxy backbones for every $\lambda\ge2$. For Hamiltonian paths, the latter threshold is $\NP$-complete for $\lambda\ge3$, while for trees hardness also holds at $\lambda+4$ for $\lambda\ge4$. On the positive side, for every fixed $\lambda\ge3$ we give linear-time constructions with $\lambda+3$ and $\lambda+4$ colors for matching and galaxy backbones, respectively. For Hamiltonian paths, we prove the upper bound $\lambda+4$ using the cycle-plus-triangles theorem, while the tree bound $\lambda+5$ follows from our general result.

\begin{table}[ht]
    \centering\footnotesize
    \setlength\tabcolsep{5pt}
    \begin{tabular}{|c|c|c|c|c|c|c|c|c|c|} 
        \hline
        $\lambda$ & $k$ & \multicolumn{2}{|c|}{$H = M$} & \multicolumn{2}{|c|}{$H = S$} & \multicolumn{2}{|c|}{$H = P$} & \multicolumn{2}{|c|}{$H = T$} \\
        \hline
        \multirow{2}{*}{$2$}
            & \multirow{2}{*}{$3$} & $\NPC$ & Thm \ref{thm:isolated-2-3} & $\NPC$ & Thm \ref{thm:isolated-2-3} & \multirow{2}{*}{$O(n)$} & \multirow{2}{*}{\cite{broersma2007backbone}} & \multirow{2}{*}{$O(n)$} & \multirow{2}{*}{\cite{broersma2007backbone}} \\
            & & $O(n)$ & \cite{broersma2009lambda} & $O(n)$ & \cite{broersma2009lambda} & & & & \\\cline{2-10}
            & $4$ & $\NPC$ & Thm \ref{thm:matching-lambda+2} & $\NPC$ & Thm \ref{thm:matching-lambda+2} & $O(n)$ & \cite{broersma2007backbone} & $O(n)$ & \cite{broersma2007backbone} \\\cline{2-10}
            & $5$ & $O(n^2)$ & \cite{araujo2022backbone} & $\NPC$ & Thm \ref{thm:galaxy-lambda+3} & ? & & $\NPC$ & Thm \ref{thm:tree-lambda+3} \\\cline{2-10}
            & $6$ & --- & & $O(n)$ & \cite{mivskuf2009backbone} & $O(n)$ & \cite{mivskuf2009backbone} & $O(n)$ & \cite{mivskuf2009backbone} \\
        \hline

        \multirow{2}{*}{$3$}
            & \multirow{2}{*}{$4$} & $\NPC$ & Thm \ref{thm:isolated-3-4} & $\NPC$ & Thm \ref{thm:isolated-3-4} & \multirow{2}{*}{$O(n)$} & \multirow{2}{*}{\cite{contributions}} & \multirow{2}{*}{$O(n)$} & \multirow{2}{*}{\cite{contributions}} \\
            & & $O(n)$ & \cite{contributions} & $O(n)$ & \cite{contributions} & & & & \\\cline{2-10}
            & $5$ & $\NPC$ & Thm \ref{thm:matching-lambda+2} & $\NPC$ & Thm \ref{thm:matching-lambda+2} & $O(n)$ & \cite{contributions} & $O(n)$ & \cite{contributions} \\\cline{2-10}
            & $6$ & $O(n)$ & Thm \ref{thm:matching-lambda+3} & $\NPC$ & Thm \ref{thm:galaxy-lambda+3} & $\NPC$ & Thm \ref{thm:path-lambda+3} & $\NPC$ & Thm \ref{thm:tree-lambda+3} \\\cline{2-10}
            & $7$ & --- & & $O(n)$ & Thm \ref{thm:galaxy-lambda+4}  & $\exists$ & Thm \ref{thm:path-lambda+4} & ? & \\\cline{2-10}
            & $8$ & --- & & --- & & --- & & $O(n)$ & Cor \ref{thm:tree-lambda+5} \\
        \hline

        $\ge 4$
            & $\lambda + 1$ & $O(n)$ & \cite{contributions} & $O(n)$ & \cite{contributions} & $O(n)$ & \cite{contributions} & $O(n)$ & \cite{contributions} \\\cline{2-10}
            & $\lambda + 2$ & $\NPC$ & Thm \ref{thm:matching-lambda+2} & $\NPC$ & Thm \ref{thm:matching-lambda+2} & $O(n)$ & \cite{contributions} & $O(n)$ & \cite{contributions} \\\cline{2-10}
            & $\lambda + 3$ & $O(n)$ & Thm \ref{thm:matching-lambda+3} & $\NPC$ & Thm \ref{thm:galaxy-lambda+3} & $\NPC$ & Thm \ref{thm:path-lambda+3} & $\NPC$ & Thm \ref{thm:tree-lambda+3} \\\cline{2-10}
            & $\lambda + 4$ & --- & & $O(n)$ & Thm \ref{thm:galaxy-lambda+4} & $\exists$ & Thm \ref{thm:path-lambda+4} & $\NPC$ & Thm \ref{thm:tree-lambda+4} \\\cline{2-10}
            & $\lambda + 5$ & --- & & --- & & --- & & $O(n)$ & Cor \ref{thm:tree-lambda+5} \\\hline
    \end{tabular}
    \caption{Summary of the complexities of the problem of finding a $\lambda$-backbone $k$-coloring for $\Delta(G) \le 4$ for different values of $\lambda$, $k$, and various classes of backbones. For $H=M,S$, double entries require every vertex to be incident with a backbone edge in the lower row and allow isolated backbone vertices in the upper row; single entries apply to both cases. Moreover, ? denotes remaining open cases; --- indicates that the case above it was already trivial (there was always such a coloring); $\exists$ denotes that a suitable coloring always exists, but no polynomial-time construction is known.}
    \label{tab:summary}
\end{table}

Note that the running times are provided for finding the coloring (if it exists). We chose this convention simply because the last cases presented in the table above are trivial when formulated as decision problems: e.g. for $\lambda \ge 3$ one can answer the question $BBC_\lambda(G, T) \le \lambda + 5$ by simply returning \texttt{YES} in constant time since it holds for every appropriate $G$ and $T$, even though finding a feasible coloring requires linear time. Similarly, deciding $BBC_\lambda(G, P) \le \lambda + 4$ is trivial, but the only known construction given here uses an exponential-time coloring algorithm (see \Cref{thm:path-lambda+4}).

Algorithms for galaxy and tree backbones also apply to matching and path backbones, respectively, while hardness results for matching and path backbones carry over to galaxies and trees.

For matching and galaxy backbones, allowing isolated vertices changes only the cases $(\lambda,k)=(2,3)$ and $(3,4)$ in \Cref{tab:summary}, by \Cref{thm:isolated-2-3,thm:isolated-3-4}. For $k \ge 5$, any coloring after deleting these vertices extends greedily in linear time, since $\Delta(G) \le 4$. All hardness results proved without isolated backbone vertices remain valid when they are allowed.

We first prove \Cref{thm:connected-lambda+5} in \Cref{sec:connected-backbone}. We then consider matching, galaxy, Hamiltonian path, and spanning tree backbones in this order. For each class, we give upper bounds and constructions before establishing the corresponding hardness results.

\section{An upper bound for connected bipartite backbones}
\label{sec:connected-backbone}

To prove the upper bound, we color the two parts of the backbone using separated palettes. The main step is to find a suitable four-coloring within these parts, using degree-list coloring and a balanced choice of vertices in an auxiliary multigraph.

\begin{proof}[Proof of \Cref{thm:connected-lambda+5}]
    The one-vertex case is immediate. Otherwise, let $A,B$ be the bipartition of the connected backbone $H$ and put $F=G[A]\cup G[B]$. Every vertex has a backbone neighbor in the opposite part, so $\deg_F(v)\le4-\deg_H(v)\le3$. We will find a proper coloring $r\colon V(F)\to\{1,2,3,4\}$ such that every vertex of color $4$ is a leaf of $H$ whose backbone neighbor has color at most $2$. This gives $r(u)+r(v)\le6$ for every $uv\in E(H)$. Then set
    \[
        c(v)=\begin{cases}
            r(v),&v\in A,\\
            \lambda+6-r(v),&v\in B.
        \end{cases}
    \]
    The palettes $\{1,2,3,4\}$ and $\{\lambda+2,\ldots,\lambda+5\}$ are disjoint for $\lambda\ge3$. Thus all edges of $G$ are properly colored, and for a backbone edge $uv$, with $u\in A$ and $v\in B$, we have $c(v)-c(u)=\lambda+6-r(v)-r(u)\ge\lambda$.

    It remains to construct $r$. Since $\Delta(F)\le3$, every copy of $K_4$ in $F$ is a connected component. Let $\mathcal K$ be the set of these components and put $R=F\setminus\bigcup_{K\in\mathcal K}V(K)$. Each vertex $x$ of such a clique has $\deg_F(x)=3$ and hence $\deg_H(x)=1$. Denote its unique backbone neighbor by $p(x)$. This neighbor has backbone degree at least $2$: otherwise $xp(x)$ would be an entire component of the connected backbone $H$, contradicting the presence of the other clique vertices. Thus $p(x)\in V(R)$ and $\deg_R(p(x))=\deg_F(p(x))\le2$.

    For each $K\in\mathcal K$, we will select one vertex $x_K$ to receive color $4$, and color its other vertices with $1,2,3$. Let $Z=\{p(x_K):K\in\mathcal K\}$. To color $R$, assign the lists
    \[
        L(v)=\begin{cases}
            \{1,2\},&v\in Z,\\
            \{1,2,3\},&v\notin Z.
        \end{cases}
    \]
    These lists satisfy $|L(v)|\ge\deg_R(v)$. We will choose the vertices $x_K$ so that $R$ can be colored from these lists.

    Recall that a connected graph is a \emph{Gallai tree} if every block is a clique or an odd cycle. By the degree-list-coloring theorem of Borodin--Erd\H{o}s--Rubin--Taylor, a connected graph that is not a Gallai tree can be colored from any lists of size at least the respective vertex degrees, in linear time \cite[Theorem~6.1]{aboulker2017colouring}. Also, in any connected graph with these list sizes, one strict inequality suffices: root a spanning tree at such a vertex and color greedily in reverse tree order. Every other vertex has an uncolored parent, and the root has a spare color.

    Consequently, only Gallai-tree components $Q$ of $R$ with minimum degree at least $2$ require care, since a vertex of degree at most $1$ always has a strict list-size inequality. Put
    \[
        D_Q=\{v\in V(Q):\deg_Q(v)=2\},\qquad t_Q=|D_Q|.
    \]
    We have $t_Q\ge3$. Indeed, if $Q$ has one block, it is an odd cycle, since $K_4$ components were removed. Otherwise, each of its at least two leaf blocks is an odd cycle: an edge would give a vertex of degree $1$, and a $K_4$ attached to another block would give a vertex of degree at least $4$. Each such leaf block contributes at least two degree-$2$ vertices. It suffices to leave one vertex of $D_Q$ outside $Z$, so that its list has size $3>\deg_Q(v)$.

    Build a bipartite multigraph $J$ whose parts correspond to $\mathcal K$ and the components of $R$. For each $x\in V(K)$, add an edge from $K$ to the component containing $p(x)$, labeled by $(x,p(x))$. Every left vertex has degree $4$. For a component $Q$ requiring care, every incident edge uses a parent $p\in D_Q$ with $\deg_H(p)=2$. Such a parent has at most one backbone neighbor in the removed cliques: if both neighbors were clique vertices, they would be leaves, and together with $p$ would form an entire component of $H$ on three vertices. This contradicts connectedness and the presence of a $K_4$. Hence these edges use distinct parents and $\deg_J(Q)\le t_Q$.

    The edges of $J$ can be colored red and blue in linear time so that every left vertex has two red edges and every right vertex $Q$ has at most $\lceil\deg_J(Q)/2\rceil$ red edges. To see this, add an auxiliary left vertex adjacent to every odd-degree right vertex. There is an even number of such vertices, since the sum of right degrees is $4|\mathcal K|$, so all degrees become even. Find an Euler circuit in each nontrivial component and color its edges alternately red and blue. Each circuit has even length, since the multigraph is bipartite, and every vertex has equally many edges of both colors. Deleting the auxiliary edges gives the required bounds.

    Choose one red edge at every left vertex $K$ and let its label determine $x_K$. For every component $Q$ requiring care, this gives
    \[
        |Z\cap D_Q|\le\left\lceil\frac{\deg_J(Q)}2\right\rceil
        \le\left\lceil\frac{t_Q}2\right\rceil\le t_Q-1,
    \]
    because $t_Q\ge3$. Thus every component of $R$ is either not a Gallai tree or has a strict list-size inequality. Color $R$ from its lists and complete each removed $K_4$ as described above. Since the cliques were components of $F$, this is a proper coloring $r$ with the required property.

    The bipartition, components, and block decompositions are found in linear time. The multigraph $J$ has four edges per removed clique, so the Euler circuits and selections also take linear time. The list-coloring procedures take linear total time, and converting $r$ to $c$ requires one pass over the vertices. Since $m \le 2n$, the total running time is $O(n)$.
\end{proof}

For $\lambda\ge4$, the bound is sharp already for spanning trees. Indeed, apply the reduction in the proof of \Cref{thm:tree-lambda+4} to the unsatisfiable formula consisting of all eight clauses on three variables with all possible choices of negations. The resulting graph $G$ with spanning tree $T$ has no $\lambda$-backbone $(\lambda+4)$-coloring, while \Cref{thm:connected-lambda+5} gives $BBC_\lambda(G,T)\le\lambda+5$.

\section{Matching backbone}

We now consider matching backbones, where the separation constraints apply to vertex-disjoint pairs. We obtain the upper bound $\lambda+3$ for $\lambda\ge3$ and prove $\NP$-completeness at $\lambda+2$ for every $\lambda\ge2$. We also examine the effect of allowing isolated backbone vertices.

\subsection{The easy cases}

We first extend the bounds of \Cref{miskuf:d,miskuf:matching} from separation $2$ to arbitrary $\lambda$ by inserting gaps between consecutive pairs of colors. This gives the following bounds for matching and $1$-degenerate backbones:
\begin{theorem}
    \label{thm:miskuf-extension}
    Let $G$ be a graph with $\Delta(G) \le 4$ and $\lambda \ge 2$. Then:
    \begin{itemize}
        \item [$(a)$] if $M$ is a matching in $G$, then $BBC_\lambda(G, M) \le 2 \lambda + 1$,
        \item [$(b)$] if $H$ is a $1$-degenerate subgraph of $G$, then $BBC_\lambda(G, H) \le 2 \lambda + 2$.
    \end{itemize}
    The respective colorings can be found in $O(n^2)$ time in $(a)$ and $O(n)$ time in $(b)$.
\end{theorem}

\begin{proof}
First, we infer from \Cref{miskuf:d,miskuf:matching} that for $\Delta(G) \le 4$ it holds that $BBC_2(G, H) \le 6$ for every $1$-degenerate backbone $H$ and $BBC_2(G, M) \le 5$ for matching backbones.

However, one can go further and for a general $\lambda$ proceed in the following way: first obtain a $2$-backbone coloring $c'$ of $G$ with backbone $H$ and then for all $v \in V(G)$ define
\begin{align*}
    c(v) = (\lambda - 2) \left\lfloor\frac{c'(v) - 1}{2}\right\rfloor + c'(v).
\end{align*}
It is sufficient to spot that such a coloring indeed maps colors $\{1, 2, 3, 4, 5, 6\}$ to $\{1, 2, \lambda + 1, \lambda + 2, 2 \lambda + 1, 2 \lambda + 2\}$ exactly in this order so that all the non-backbone edge requirements are trivially preserved between $c'$ and $c$.
Writing $f(i) = (\lambda - 2)\lfloor(i - 1)/2\rfloor + i$, we have $f(i + 2) - f(i) = \lambda$ and $f$ is strictly increasing. Hence $|c'(u) - c'(v)| \ge 2$ implies $|c(u) - c(v)| \ge \lambda$ for every edge $uv \in E(H)$.

The running times follow directly from the respective complexities for finding the colorings from \Cref{miskuf:d,miskuf:matching}, as the  transformation described above runs in linear time.
\end{proof}

For matching backbones and $\lambda\ge3$, we can improve the bound $2\lambda+1$ to $\lambda+3$. To do so, we combine the coloring approach from \cite{broersma2007backbone,broersma2009lambda} with the following algorithmic form of Brooks' theorem, which will also be used for galaxy backbones:
\begin{theorem}[\cite{lovasz1975three,baetz2014brooks}]
    \label{Observation:Brooks}
    For any connected graph $G$ it holds that either it is a clique or an odd cycle, or $\chi(G) \le \Delta(G)$. We can distinguish these cases and find the appropriate coloring in both cases in $O(m + n)$ time.
\end{theorem}

\begin{theorem}
    \label{thm:matching-lambda+3}
    Let $G$ be a graph with $\Delta(G) \le 4$ and $M$ be a matching in $G$.
    Then for $\lambda \ge 3$ it holds that $BBC_\lambda(G, M) \le \lambda + 3$. A suitable coloring can be found in $O(n)$ time.
\end{theorem}

\begin{proof}
    It suffices to color each connected component separately. From \Cref{Observation:Brooks} we have three subcases: either $G$ is an odd cycle, a complete graph (on at most $5$ vertices), or neither of these.
    In the first case, write the odd cycle as $v_1,\ldots,v_n,v_1$ with $v_nv_1 \notin E(M)$. Set $c(v_n)=2$, color all other odd-indexed vertices with $1$, and all even-indexed vertices with $\lambda+2$.
    In the second case, extend the clique to $K_5$ and extend its matching to two disjoint edges. Color their endpoints with $1, \lambda + 1$ and $2, \lambda + 2$, respectively, and the remaining vertex with $3$. For $\lambda \ge 3$ these five colors are distinct and each matching edge has color difference $\lambda$. Restricting this coloring to the original clique gives the required coloring in $O(1)$ time.

    In the remaining case, obtain a proper $4$-coloring $f$ using \Cref{Observation:Brooks}. Following \cite{broersma2009lambda}, for each $uv \in M$ with $f(u)=i<j=f(v)$, set $c(u)=i$ and $c(v)=\lambda+j-1$. The palettes of the four classes are $\{1\}$, $\{2,\lambda+1\}$, $\{3,\lambda+2\}$, and $\{\lambda+3\}$, which are disjoint for $\lambda\ge3$. Every matching difference is at least $\lambda$. Greedily color unmatched vertices afterwards, since the palette has at least five colors and $\Delta(G)\le4$. This takes linear time.
\end{proof}

\subsection{The hard cases}

For each fixed $\lambda$ and color bound $k$, the decision problems considered here belong to $\NP$: a coloring $c\colon V(G) \to \{1, \ldots, k\}$ is a certificate whose edge constraints can be checked in $O(m + n)$ time. In the reductions below, gadget copies are initially vertex-disjoint, and every backbone edge is also an edge of the constructed graph.

We will repeatedly use the following observation: if $\lambda \ge r$, every vertex incident with a backbone edge in a $\lambda$-backbone $(\lambda+r)$-coloring has a color in one of the disjoint palettes $\{1,\ldots,r\}$ and $\{\lambda+1,\ldots,\lambda+r\}$, and every backbone edge joins the two palettes.

We use this observation to construct gadgets that force prescribed pairs of colors. The first two gadgets are shown in \Cref{fig:gadgetXY}:
\begin{figure}[H]
\centering
\begin{tikzpicture}[scale=1]
    \node[draw, circle, fill=black, inner sep=2pt] (AA) at (1,2) {};
    \node (AA1) at (1,2.3) {$v_5$};
    \node[draw, circle, fill=black, inner sep=2pt] (BA) at (2,2) {};
    \node (BA1) at (2,2.3) {$v_6$};
    \node[draw, circle, fill=black, inner sep=2pt] (CA) at (3,2) {};
    \node (CA1) at (3,2.3) {$v_7$};
    \node[draw, circle, fill=black, inner sep=2pt] (DA) at (4,2) {};
    \node (DA1) at (4,2.3) {$v_8$};
    \node[draw, circle, fill=black, inner sep=2pt] (AB) at (1,1) {};
    \node (AB1) at (1,0.7) {$v_1$};
    \node[draw, circle, fill=black, inner sep=2pt] (BB) at (2,1) {};
    \node (BB1) at (2,0.7) {$v_2$};
    \node[draw, circle, fill=black, inner sep=2pt] (CB) at (3,1) {};
    \node (CB1) at (3,0.7) {$v_3$};
    \node[draw, circle, fill=black, inner sep=2pt] (DB) at (4,1) {};
    \node (DB1) at (4,0.7) {$v_4$};
    \node[draw, circle, fill=black, inner sep=2pt] (P) at (1.5,4) {};
    \node (P1) at (1.5,3.7) {$v_9$};
    \node[draw, circle, fill=black, inner sep=2pt] (Q) at (3.5,4) {};
    \node (Q1) at (3.5,3.7) {$v_{10}$};
    \node[draw, circle, fill=black, inner sep=2pt] (X) at (1.5,5) {};
    \node (X1) at (1.5,5.3) {$v_{11}$};
    \node[draw, circle, fill=black, inner sep=2pt] (Y) at (3.5,5) {};
    \node (Y1) at (3.5,5.3) {$v_{12}$};
    \draw (AA) -- (BA);
    \draw (AA) -- (AB) [line width=3pt];
    \draw (AA) -- (BB);
    \draw (AB) -- (BB);
    \draw (AB) -- (BA);
    \draw (BA) -- (BB);
    \draw (BB) -- (CB) [line width=3pt];
    \draw (BA) -- (CA) [line width=3pt];
    \draw (CA) -- (DA);
    \draw (CB) -- (DB);
    \draw (CA) -- (DB);
    \draw (CB) -- (DA);
    \draw (DA) -- (DB) [line width=3pt];
    \draw (AA) to[out=150,in=180] (P);
    \draw (AB) to[out=150,in=180] (P);
    \draw (DA) to[out=30,in=0] (Q);
    \draw (DB) to[out=30,in=0] (Q);
    \draw (P) -- (Q);
    \draw (P) -- (X) [line width=3pt];
    \draw (Q) -- (Y) [line width=3pt];
    \draw (CA) to[out=128,in=180] (Y);
    \draw (CB) to[out=128,in=180] (Y);
\end{tikzpicture}
\quad
\begin{tikzpicture}[xscale=0.6]
    \node (P) at (1.5,4) {};
    \node (Q) at (3.5,4) {};
    \node (QL) at (3,4) {};
    \node (QR) at (4,4) {};
    \node[draw, circle, fill=black, inner sep=2pt] (X) at (1.5,5) {};
    \node (X1) at (0.7,4.95) {$v_{11}$};
    \node[draw, circle, fill=black, inner sep=2pt] (Y) at (3.5,4.8) {};
    \node (Y1) at (3.5,5.1) {$v_{12}$};

    \node (Q2) at (6.5,4) {};
    \node (P2) at (8.5,4) {};
    \node (Q2L) at (6,4) {};
    \node (Q2R) at (7,4) {};
    \node[draw, circle, fill=black, inner sep=2pt] (Y2) at (6.5,4.8) {};
    \node (X12) at (6.5,5.15) {$v'_{12}$};
    \node[draw, circle, fill=black, inner sep=2pt] (X2) at (8.5,5) {};
    \node (Y12) at (9.3,5) {$v'_{11}$};
    
    \node[draw, circle, fill=black, inner sep=2pt] (T1) at (1,7) {};
    \node (T11) at (1,7.3) {$u_1$};
    \node[draw, circle, fill=black, inner sep=2pt] (T2) at (4,7) {};
    \node (T21) at (4,7.3) {$u_2$};
    \node[draw, circle, fill=black, inner sep=2pt] (T3) at (6,7) {};
    \node (T31) at (6,7.3) {$u_3$};
    \node[draw, circle, fill=black, inner sep=2pt] (T4) at (9,7) {};
    \node (T41) at (9,7.3) {$u_4$};
    \draw (T1) -- (T2) [line width=3pt];
    \draw (T3) -- (T4) [line width=3pt];

    \draw (Y) -- (Y2);
    
    \draw (T1) -- (X) -- (T2);
    \draw (T1) -- (X2) -- (T3);
    \draw (T4) -- (X);
    \draw (T4) -- (X2);

    \draw (P) -- (X) [line width=3pt];
    \draw (Q) -- (Y) [line width=3pt];
    \draw (QL) -- (Y) -- (QR);
    \draw (P2) -- (X2) [line width=3pt];
    \draw (Q2) -- (Y2) [line width=3pt];
    \draw (Q2L) -- (Y2) -- (Q2R);

    \draw (0.5, 3) rectangle (4.5, 4.12);
    \node at (2.5, 3.5) {$X$};
    \draw (5.5, 3) rectangle (9.5, 4.12);
    \node at (7.5, 3.5) {$X$};
\end{tikzpicture}
\caption{Gadgets $X$ (left) and $Y$ (right) with their backbone edges in bold.}
\label{fig:gadgetXY}
\end{figure}
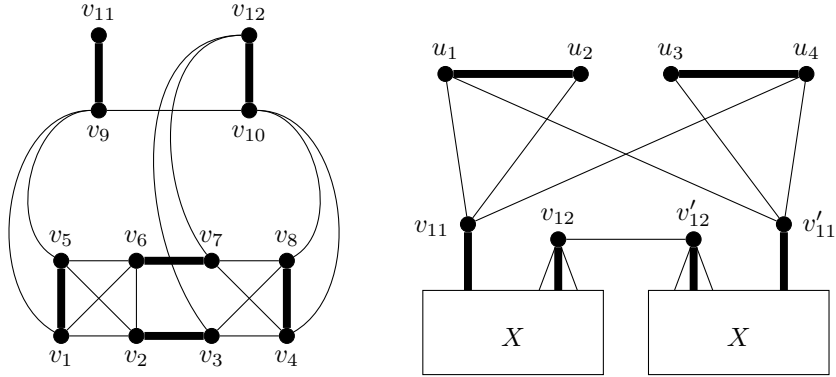

The whole idea of this proof would revolve around forcing proper pairs of colors on certain pairs of vertices. Thus, the following notion is handy:
\begin{definition}
    Let $G$ be a graph and let $c\colon V(G) \to \N_+$ be its coloring.
    For distinct $u, v \in V(G)$, let $c(\{u, v\}) = \{c(u), c(v)\}$ denote the set of their colors. Thus, $c(\{u, v\}) = \{a, b\}$, where $a \ne b$, allows either assignment of these two colors to $u$ and $v$.
\end{definition}

\begin{lemma}
    \label{lem:gadgetX}
    Let $G$ with backbone $H$ be the gadget $X$ presented in \Cref{fig:gadgetXY}.
    For any $\lambda \ge 2$ and for every $\lambda$-backbone $(\lambda + 2)$-coloring $c$ it holds that $c(\{v_{11}, v_{12}\}) = \{1, \lambda + 2\}$.
\end{lemma}

\begin{proof}
    Let us proceed with considering the only three feasible colorings of an edge $v_1v_5 \in E(H)$:
    \begin{enumerate}
        \item $c(\{v_1, v_5\}) = \{1, \lambda + 1\}$, which implies $c(\{v_2, v_6\}) = \{2, \lambda + 2\}$, because all these vertices form a clique in $G$ -- and if we have a $K_4$ subgraph with all vertices incident to at least one backbone edge, the only available colors are $\{1, 2, \lambda + 1, \lambda + 2\}$.
        Now we have two subcases:
        \begin{enumerate}
            \item either $c(\{v_3, v_7\}) = \{1, \lambda + 2\}$, and therefore $c(\{v_4, v_8\}) = \{2, \lambda + 1\}$ -- a contradiction, since $v_4v_8 \in E(H)$.
            \item or $c(\{v_3, v_7\}) = \{2, \lambda + 2\}$ and thus $c(\{v_4, v_8\}) = \{1, \lambda + 1\}$.
        \end{enumerate}
        The latter case guarantees that $c(v_{10}) \in \{2, \lambda + 2\}$, and either way $c(v_{12}) \ne \lambda + 1$.
        Since $v_{12}v_3, v_{12}v_7 \in E(G)$, it follows that $c(v_{12}) \notin \{2, \lambda + 2\}$, so it has to be the case that $c(v_{12}) = 1$.

        It is easy to see that this also implies the coloring of the remaining vertices: $c(v_{10}) = \lambda + 2$, $c(v_{9}) = 2$, and $c(v_{11}) = \lambda + 2$.

        \item $c(\{v_1, v_5\}) = \{2, \lambda + 2\}$. This case is symmetric to the previous one; we just substitute every color $i$ with $\lambda + 3 - i$; therefore we obtain $c(v_{11}) = 1$ and $c(v_{12}) = \lambda + 2$.

        \item $c(\{v_1, v_5\}) = \{1, \lambda + 2\}$. This implies that $c(\{v_2, v_6\}) = \{2, \lambda + 1\}$, because these vertices form a clique in $G$. Thus, $c(\{v_3, v_7\}) = \{1, \lambda + 2\}$, and $c(\{v_4, v_8\}) = \{2, \lambda + 1\}$ -- but the last formula contradicts the fact that $v_4v_8 \in E(H)$.
    \end{enumerate}
    Both terminal orientations are realizable: the vector $(c(v_1), \ldots, c(v_{12})) = (1, \lambda + 2, 2, 1, \lambda + 1, 2, \lambda + 2, \lambda + 1, 2, \lambda + 2, \lambda + 2, 1)$ satisfies every edge constraint, and reflection $i \mapsto \lambda + 3 - i$ gives the other orientation.
\end{proof}

\begin{lemma}
    \label{lem:gadgetY}
    Let $G$ with backbone $H$ be the gadget $Y$ presented in \Cref{fig:gadgetXY}.
    For any $\lambda \ge 2$ and for every $\lambda$-backbone $(\lambda + 2)$-coloring $c$ it holds that:
    \begin{itemize}
        \item[$(a)$] either $c(u_1) = 2$, $c(u_2) = \lambda + 2$, $c(u_3) = 1$, $c(u_4) = \lambda + 1$,
        \item[$(b)$] or $c(u_1) = \lambda + 1$, $c(u_2) = 1$, $c(u_3) = \lambda + 2$, $c(u_4) = 2$.
    \end{itemize}
\end{lemma}

\begin{proof}
    From \Cref{lem:gadgetX} we know that $c(\{v_{11}, v_{12}\}) = c(\{v'_{11}, v'_{12}\}) = \{1, \lambda + 2\}$. Moreover, since $v_{12}v'_{12} \in E(G)$ it has to be the case that $c(\{v_{12}, v'_{12}\}) = \{1, \lambda + 2\}$ -- and therefore $c(\{v_{11}, v'_{11}\}) = \{1, \lambda + 2\}$ as well.

    Assume now that $c(v_{11}) = 1$ and $c(v'_{11}) = \lambda + 2$. Then, $c(\{u_1, u_2\}) = \{2, \lambda + 2\}$ and $c(\{u_3, u_4\}) = \{1, \lambda + 1\}$.
    And since $v'_{11}u_1, v_{11}u_4 \in E(G)$ it has to be the case that $c(u_1) = 2$, $c(u_2) = \lambda + 2$, $c(u_3) = 1$, $c(u_4) = \lambda + 1$, that is, exactly the case $(a)$.

    The case $c(v_{11}) = \lambda + 2$ and $c(v'_{11}) = 1$ follows analogously, giving $(b)$. Conversely, the colorings of $X$ supplied in the preceding proof, together with the stated colors of $u_1, \ldots, u_4$, realize both $(a)$ and $(b)$.
\end{proof}

\begin{figure}[H]
\centering
\begin{tikzpicture}[scale=1]
    \node[draw, circle, fill=black, inner sep=2pt] (A) at (1,4) {};
    \node (A1) at (1,4.3) {$a_1$};
    \node[draw, circle, fill=black, inner sep=2pt] (B) at (2,3) {};
    \node (B1) at (2,3.3) {$a_2$};
    \node[draw, circle, fill=black, inner sep=2pt] (C) at (3,4) {};
    \node (C1) at (3,4.3) {$a_3$};

    \node[draw, circle, fill=black, inner sep=2pt] (D) at (-1,4) {};
    \node (D1) at (-1,4.3) {$b_1$};
    \node[draw, circle, fill=black, inner sep=2pt] (E) at (2,2) {};
    \node (E1) at (2.3,2.3) {$b_2$};
    \node[draw, circle, fill=black, inner sep=2pt] (F) at (5,4) {};
    \node (F1) at (5,4.3) {$b_3$};

    \node[draw, circle, fill=black, inner sep=2pt] (X11) at (-1.5,3) {};
    \node (Y11) at (-1.8,3.1) {$u_1$};
    \node[draw, circle, fill=black, inner sep=2pt] (X14) at (-0.5, 3) {};
    \node (Y14) at (-0.2,3.1) {$u_3$};

    \node (P1) at (-1.5,2.38) {};
    \node (P1L) at (-1.8,2.38) {};
    \node (P1R) at (-1.2,2.38) {};
    \node (Q1) at (-0.5,2.38) {};
    \node (Q1L) at (-0.8,2.38) {};
    \node (Q1R) at (-0.2,2.38) {};

    \node[draw, circle, fill=black, inner sep=2pt] (X21) at (4.5,3) {};
    \node (Y21) at (4.2,3.1) {$u''_1$};
    \node[draw, circle, fill=black, inner sep=2pt] (X24) at (5.5, 3) {};
    \node (Y24) at (5.8,3.1) {$u''_3$};

    \node (P2) at (4.5,2.38) {};
    \node (P2L) at (4.2,2.38) {};
    \node (P2R) at (4.8,2.38) {};
    \node (Q2) at (5.5,2.38) {};
    \node (Q2L) at (5.2,2.38) {};
    \node (Q2R) at (5.8,2.38) {};

    \node[draw, circle, fill=black, inner sep=2pt] (X31) at (1.5,1) {};
    \node (Y31) at (1.2,1.1) {$u'_1$};
    \node[draw, circle, fill=black, inner sep=2pt] (X34) at (2.5, 1) {};
    \node (Y34) at (2.8,1.1) {$u'_3$};

    \node (P3) at (1.5,0.38) {};
    \node (P3L) at (1.2,0.38) {};
    \node (P3R) at (1.8,0.38) {};
    \node (Q3) at (2.5,0.38) {};
    \node (Q3L) at (2.2,0.38) {};
    \node (Q3R) at (2.8,0.38) {};

    \draw (-2, 2.5) rectangle (0, 1.5);
    \node at (-1, 2) {$Y_i$};
    \draw (1, 0.5) rectangle (3, -0.5);
    \node at (2, 0) {$Y_j$};
    \draw (4, 2.5) rectangle (6, 1.5);
    \node at (5, 2) {$Y_k$};

    \draw (A) -- (B) -- (C) -- (A);
    \draw (A) -- (D) [line width=3pt];
    \draw (B) -- (E) [line width=3pt];
    \draw (C) -- (F) [line width=3pt];

    \draw (X11) -- (D) -- (X14);
    \draw (P1) -- (X11) [line width=3pt];
    \draw (P1L) -- (X11) -- (P1R);
    \draw (Q1) -- (X14) [line width=3pt];
    \draw (Q1L) -- (X14) -- (Q1R);

    \draw (X21) -- (F) -- (X24);
    \draw (P2) -- (X21) [line width=3pt];
    \draw (P2L) -- (X21) -- (P2R);
    \draw (Q2) -- (X24) [line width=3pt];
    \draw (Q2L) -- (X24) -- (Q2R);

    \draw (X31) -- (E) -- (X34);
    \draw (P3) -- (X31) [line width=3pt];
    \draw (P3L) -- (X31) -- (P3R);
    \draw (Q3) -- (X34) [line width=3pt];
    \draw (Q3L) -- (X34) -- (Q3R);
\end{tikzpicture}
\caption{Clause gadget.}
\label{fig:gadgetK}
\end{figure}
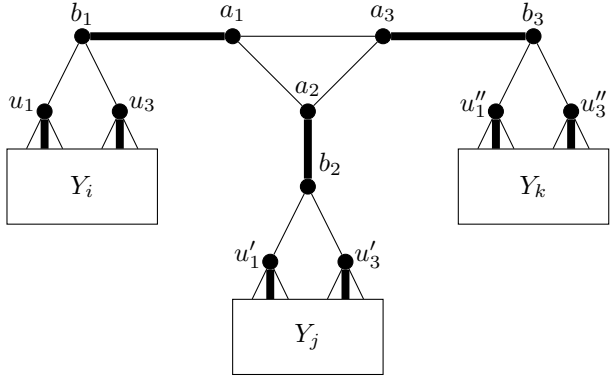

Now we may proceed with the main theorem:
\begin{theorem}
    \label{thm:matching-lambda+2}
    Let $G$ be a graph with $\Delta(G) \le 4$ and $M$ be a matching in $G$. Then for every $\lambda \ge 2$ deciding whether $BBC_\lambda(G, M) \le \lambda + 2$ is $\NP$-complete.
\end{theorem}

\begin{proof}
    We proceed by reducing instances of a well-known $\NP$-complete problem \texttt{NAE-3-SAT} with $n$ variables $x_1, \dots, x_n$ and $m$ $3$-CNF clauses $C_1, \dots, C_m$ \cite{garey-johnson}.
    For simplicity, we will use the version of \texttt{NAE-3-SAT} in which all literals appear without negation \cite{schaefer1978complexity}.

    We construct gadget $Y_i$, representing variable $x_i$ (appearing $k_i$ times in the input formula), as a sequence of $k_i$ subgadgets $Y$ such that vertex $u_3$ of the $j$-th subgadget is connected with vertex $u_2$ of the $(j + 1)$-th subgadget by a non-backbone edge. By \Cref{lem:gadgetY}, opposite states in consecutive copies would give these endpoints equal colors. Thus all copies have the same state.
    This increases the degree of each $u_2$ and $u_3$ by at most $1$, so their degrees are at most $3$.

    A gadget for clauses is present in \Cref{fig:gadgetK}. Each such gadget is connected to $u_1$ and $u_3$ for appropriate gadgets $Y_i$ representing variables $x_i$ in this clause.
    In particular, if this is the $j$-th appearance of $x_i$ in the whole formula, we connect it to the vertices from the $j$-th subgadget of $Y_i$.

    Note that for any clause gadget all colors assigned to $a_1, a_2, a_3$ have to be distinct. This is possible in a $\lambda$-backbone $(\lambda + 2)$-coloring only if there are in $\{b_1, b_2, b_3\}$ both vertices colored with $\{1, 2\}$ and colored with $\{\lambda + 1, \lambda + 2\}$.

    Now let us proceed with the correctness of the reduction.
    Suppose there exists a $\lambda$-backbone $(\lambda + 2)$-coloring $c$ for a graph built for a given formula. From \Cref{lem:gadgetY} we know that $c(u_1) \in \{2, \lambda + 1\}$ for any vertex $u_1$ in any gadget $Y_j$.
    We define the assignment $\phi$ by setting $\phi(x_j) = 1$ when all vertices $u_1$ in $Y_j$ have color $2$ in $c$, otherwise we set $\phi(x_j) = 0$.
    Now, if there were a clause with all literals set to the identical value, then all $c(\{u_1, u_3\}) = \{1, 2\}$ or all $c(\{u_1, u_3\}) = \{\lambda + 1, \lambda + 2\}$ for all vertices adjacent to this clause gadget. Therefore, $\{c(b_1), c(b_2), c(b_3)\}$ has to be either a subset of $\{1, 2\}$, or a subset of $\{\lambda + 1, \lambda + 2\}$. But then one could not find a valid coloring for vertices $a_i$ in the gadget -- a contradiction. Therefore, each clause has at least one true and at least one false literal for the assignment $\phi$.

    Suppose now that we have an assignment $\phi$ satisfying the input formula. Then, it is sufficient to color each gadget $Y_j$ starting from assigning $c(u_1) = 2$ for all its vertices $u_1$, when $\phi(x_j) = 1$. Otherwise, we set $c(u_1) = \lambda + 1$. By \Cref{lem:gadgetX,lem:gadgetY} one can easily verify that such a coloring is feasible for all $Y_j$.
    For each clause, assign $a_i$ a color from the same pair, $\{1, 2\}$ or $\{\lambda + 1, \lambda + 2\}$, as its corresponding $u_1,u_3$. Since each pair is required by at most two of the three vertices $a_i$, their colors can be chosen distinct. Set $c(b_i) = c(a_i) + \lambda$ in the first case and $c(b_i) = c(a_i) - \lambda$ in the second. Then $|c(a_i) - c(b_i)| = \lambda$, and $b_i$ differs from both its neighbors $u_1,u_3$.

    For completeness, let us note that gadgets in \Cref{fig:gadgetXY,fig:gadgetK} have maximum degree at most $4$. Each literal occurrence uses a distinct pair $u_1,u_3$ in its variable gadget, increasing their degrees by $1$. Hence their final degrees do not exceed $4$. The backbone consists of the disjoint gadget matchings and edges $a_ib_i$, so it is a matching. Since $\sum_i k_i = 3m$ and each gadget has constant size, the construction has $O(n + m)$ vertices and edges and takes polynomial time.
\end{proof}

Finally, allowing isolated backbone vertices gives two additional hard cases.
\begin{theorem}
    \label{thm:isolated-2-3}
    Let $G$ be a graph with $\Delta(G) \le 4$ and $H$ be a matching or galaxy backbone with isolated vertices allowed. Then deciding whether $BBC_2(G,H) \le 3$ is $\NP$-complete.
\end{theorem}

\begin{proof}
    We reduce from $3$-colorability of graphs with maximum degree at most $4$, which is $\NP$-complete \cite{kochol2003colorability}. For an input graph $G$, take $H=(V(G),\emptyset)$. Then $BBC_2(G,H) \le 3$ if and only if $\chi(G) \le 3$. If a nonempty backbone is required, append a disjoint copy of $K_2$ and make its edge the only backbone edge; its endpoints can always receive colors $1$ and $3$. The backbone is a matching and hence also a galaxy, and the maximum degree remains at most $4$.
\end{proof}

\begin{theorem}
    \label{thm:isolated-3-4}
    Let $G$ be a graph with $\Delta(G) \le 4$ and $H$ be a matching or galaxy backbone with isolated vertices allowed. Then deciding whether $BBC_3(G,H) \le 4$ is $\NP$-complete.
\end{theorem}

\begin{proof}
    We proceed with a reduction from \texttt{NAE-4-SAT} in which all literals appear without negation \cite{schaefer1978complexity}. Given variables $x_1,\ldots,x_n$ and clauses $C_1,\ldots,C_m$, we ask whether there exists an assignment giving each clause at least one true and one false literal. Discard unused variables. For each variable $x_i$ with $k_i$ occurrences, create a path $p_{i,1},\ldots,p_{i,2k_i}$ and put the edges $p_{i,2j-1}p_{i,2j}$, for $j=1,\ldots,k_i$, into the backbone $M$. The remaining path edges are non-backbone edges.

    For each clause, create a copy of $K_4$ and join each of its vertices to the vertex $p_{i,2j-1}$ representing the corresponding $j$-th occurrence of $x_i$. All clique edges and these joining edges are non-backbone edges. Thus $M$ is a matching with isolated clause vertices. Variable-path vertices have degree at most $3$ and clause vertices have degree $4$. Since $\sum_i k_i=4m$, the construction has $12m$ vertices and takes $O(n+m)$ time.

    In a $3$-backbone $4$-coloring, every backbone edge has endpoint colors $1$ and $4$. Proper coloring of the remaining path edges forces alternation along each variable path, so all its odd-indexed vertices have the same color. These colors encode a truth assignment, with $4$ representing true and $1$ representing false. Conversely, every truth assignment gives such alternating colorings of the variable paths.

    A clause clique is colorable if and only if its four variable neighbors include both colors $1$ and $4$. If they all have the same color, only three colors remain available for its four pairwise adjacent vertices. Conversely, if both colors occur, assign color $1$ to a clique vertex whose variable neighbor has color $4$, and color $4$ to one whose variable neighbor has color $1$. Assign colors $2$ and $3$ to the other two clique vertices. This satisfies every edge incident with the clause clique.

    Therefore, a valid backbone coloring exists if and only if the input formula is satisfiable. This proves hardness for matching backbones with isolated vertices, and hence also for galaxy backbones with isolated vertices.
\end{proof}

\section{Galaxy backbone}

Every matching is a galaxy, so the hardness result at $\lambda+2$ carries over from the previous section. We now prove the upper bound $\lambda+4$ for galaxy backbones when $\lambda\ge3$ and establish $\NP$-completeness at $\lambda+3$ for every $\lambda\ge2$. For $\lambda=2$, the upper bound follows from \Cref{thm:miskuf-extension}, since $2\lambda+2=\lambda+4$.

For a coloring $c$ of $G$, let $C_i=\{v\in V(G)\colon c(v)=i\}$ denote the color class of $i$. We will also use this notation for path backbones.

\subsection{The easy cases}

Let us first improve on \Cref{thm:miskuf-extension}, which gave a bound $BBC_\lambda(G, S) \le 2 \lambda + 2$:
\begin{theorem}
    \label{thm:galaxy-lambda+4}
    Let $G$ be a graph with $\Delta(G) \le 4$ and $S$ be its spanning galaxy.
    Then for $\lambda \ge 3$ it holds that $BBC_\lambda(G, S) \le \lambda + 4$. A suitable coloring can be found in $O(n)$ time.
\end{theorem}

\begin{proof}
    Delete any vertices isolated in $S$ and color them greedily afterwards. We work componentwise and use \Cref{Observation:Brooks} to distinguish $K_5$, odd cycles, and the remaining $4$-colorable graphs.

    The first one is obvious, since we can color the roots of the stars with the lowest colors starting from $1$ and their leaves with distinct colors starting from $\lambda + 1$, processing the stars in the same order -- and since there is at least one star in $G$, the maximum color does not exceed $\lambda + n - 1 = \lambda + 4$.

    In the second case we can pick any non-backbone edge and color its endpoints with $1$ and $2$, and then assign colors $2$ and $\lambda + 2$ alternately along the path.

    Otherwise, if $\lambda\ge4$, obtain a proper $4$-coloring and give its classes the respective palettes
    \[
        \{1,\lambda+2\},\quad \{2,\lambda+1\},\quad
        \{3,\lambda+4\},\quad \{4,\lambda+3\}
    \]
    as in \cite{broersma2009lambda}. Roots in classes $1,2$ take the lower color; roots in classes $3,4$ take the upper color. Leaves take the opposite (upper or lower) choice in their own class. The palettes are disjoint, and every backbone difference is at least $\lambda$.

    Unfortunately, for $\lambda = 3$, in \cite{broersma2009lambda} there is only a weaker result (Theorem 4(d)) that establishes $BBC_3(G, S) \le 8$. However, we can improve it by considering an induction over stars in the backbone $S$. Let $c_i$ be the partial coloring of the first $i$ stars $S_1$, $S_2$, \ldots, $S_i$, and let $F_i(v) = \{c_{i - 1}(u)\colon uv \in E(G) \land u \in V(S_1) \cup \ldots \cup V(S_{i - 1})\}$ be a set of colors forbidden for any $v \in V(S_i)$ by $c_{i - 1}$.
    The base case is the empty partial coloring $c_0$. Suppose now that $c_{i - 1}$ is a valid partial $3$-backbone $7$-coloring for the graph $G[V(S_1) \cup \ldots \cup V(S_{i - 1})]$ with backbone $S_1 \cup \ldots \cup S_{i - 1}$. Then, we color the $i$-th star with a root $v \in V(S_i)$ in the following way:
    \begin{enumerate}
        \item if $1 \notin F_i(v)$, then we set $c_i(v) = 1$. Next, we color the leaves $u$ of $S_i$ sequentially by noticing that there is always at least one color $k \in \{4, 5, 6, 7\}$ available for each $u$, because $u$ has at most $3$ already colored neighbors other than $v$ (implied by $uv \in E(S)$ and $\Delta(G) \le 4$), and no more backbone edges, so we can color them sequentially.
        \item if $7 \notin F_i(v)$, then the procedure is identical to the one above, only with $k \in \{1, 2, 3, 4\}$ instead.
        \item if $1, 7 \in F_i(v)$ and $\deg_S(v) \ge 3$, then it had to be the case that $v$ has at least two neighbors outside of $S_i$ so $\deg_G(v) \ge \deg_S(v) + 2 \ge 5$ -- impossible.
        \item if $1, 7 \in F_i(v)$ and $\deg_S(v) = 2$, then $F_i(v) = \{1, 7\}$. Let $u_1,u_2$ be the leaves.
        If $u_1u_2 \notin E(G)$, the choice $c_i(v) = 2$ fails only when some $F_i(u_j) = \{5,6,7\}$, and the choice $c_i(v) = 6$ fails only when some $F_i(u_j) = \{1,2,3\}$. If both fail, these are different leaves; assign them colors $1$ and $7$, respectively, and set $c_i(v) = 4$.
        If $u_1u_2 \in E(G)$, then $|F_i(u_j)| \le 2$. With $c_i(v) = 2$, each leaf has a nonempty list $\{5,6,7\} \setminus F_i(u_j)$. Distinct colors can be chosen unless both lists are the same singleton. In that case both forbidden sets are contained in $\{5,6,7\}$, so $c_i(v) = 6$, $c_i(u_1) = 1$, $c_i(u_2) = 2$ is valid.
        \item if $1, 7 \in F_i(v)$ and $\deg_S(v) = 1$, then for the other vertex $u$ of this star it has to be the case that $1, 7 \in F_i(u)$ -- otherwise we could make $u$ the root of $S_i$ instead. Each endpoint forbids at most one further color in $\{2,3,4,5,6\}$. Thus at most two of the four ordered pairs $(2,5)$, $(3,6)$, $(5,2)$, $(6,3)$ are excluded for $(c_i(v),c_i(u))$, and any remaining pair is valid.
    \end{enumerate}
    Of course for all vertices $w \in V(S_1) \cup \ldots \cup V(S_{i - 1})$ we set $c_i(w) = c_{i - 1}(w)$. This preserves all previous constraints and satisfies every edge within $S_i$ or joining it to a previous star. Each star has at most five vertices, so all choices take constant time per star and the whole procedure runs in $O(n)$ time.
\end{proof}

\subsection{The hard case}

To prove hardness at $\lambda+3$, we first introduce a gadget that forces three terminals to have the same color:
\begin{figure}[H]
\centering
\begin{tikzpicture}[scale=0.95]
    \node[draw, circle, fill=black, inner sep=2pt, label={-90:$w$}] (X) at (0, 4) {};
    \foreach \val [count=\i from 1] in {-4, 0, 4} {
        \ifthenelse{\i = 1}{\def\x{180}}{\def\x{0}};
        \node[draw, circle, fill=black, inner sep=2pt, label={\x:$u_{\i}$}] (A\i) at (\val, 5) {};
        \node[draw, circle, fill=black, inner sep=2pt] (B\i) at (\val - 1, 6) {};
        \node[draw, circle, fill=black, inner sep=2pt] (C\i) at (\val + 1, 6) {};
        \node[draw, circle, fill=black, inner sep=2pt] (D\i) at (\val, 7) {};
        \node[draw, circle, fill=black, inner sep=2pt, label={90:$v_{\i}$}] (Y\i) at (\val, 8) {};
        \draw (A\i) -- (B\i) -- (C\i) -- (D\i) -- (A\i) -- (C\i);
        \draw (B\i) -- (D\i);
        \draw (A\i) -- (X) [line width=3pt];
        \draw (B\i) -- (Y\i) -- (C\i) [line width=3pt];
        \draw (D\i) -- (Y\i) [line width=3pt];
    }
\end{tikzpicture}
\caption{$W$-gadget.}
\label{fig:GalaxyXGadget}
\end{figure}
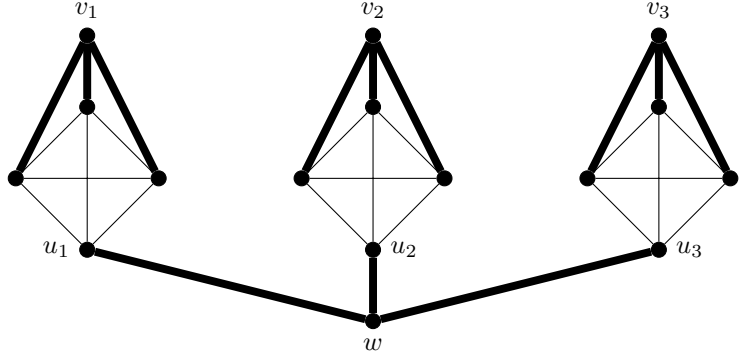

\begin{lemma}\label{lem:SG}
    Let $\lambda \ge 3$, and let $G$ and $S$ be a $W$-gadget, that is, a graph and its backbone as presented in \Cref{fig:GalaxyXGadget}.
    Then, every $\lambda$-backbone $(\lambda + 3)$-coloring $c$ of $G$ with backbone $S$ has $c(v_1) = c(v_2) = c(v_3) \in \{1, \lambda + 3\}$.
\end{lemma}

\begin{proof}
    First, we note that every $v_i$ in any $\lambda$-backbone $(\lambda + 3)$-coloring cannot get any color other than $1$ or $\lambda + 3$ in $c$; otherwise its neighborhood cannot be colored properly -- as there would remain fewer than $3$ colors at a distance at least $\lambda$ from $c(v_i)$.

    Additionally, suppose that some $v_i$ and $v_j$ have different colors in $c$. Without loss of generality, we assume that $c(v_i) = 1$ and $c(v_j) = \lambda + 3$. Then, $c(u_i) \in \{1, \ldots, \lambda\}$ and $c(u_j) \in \{4, \ldots, \lambda + 3\}$, so the backbone edges $u_iw$ and $u_jw$ require $c(w) \ge \lambda + 1$ and $c(w) \le 3$, respectively. Since $\lambda \ge 3$, this is a contradiction.
\end{proof}

For $\lambda=2$, we need a modification of the $W$-gadget. Denote its three unnamed clique vertices adjacent to $v_i$ by $a_i,b_i,d_i$. Introduce fresh vertices $z,v'_j,u'_j,a'_j,b'_j,d'_j$, for $j=1,2$, where $z$ is shared by the two branches. Make each set $\{u'_j,a'_j,b'_j,d'_j\}$ a clique, add the backbone edges $zu'_j,v'_ja'_j,v'_jb'_j,v'_jd'_j$, and add the non-backbone edges $v'_1v'_2$ and $wz$. Call the resulting graph and backbone an $R$-gadget (see \Cref{fig:GalaxyRGadget}). It has $27$ vertices and its backbone consists of seven stars. Its vertices $v_1,v_2,v_3,z$ have degree $3$, and all other vertices have degree $4$.

\begin{figure}[H]
\centering
\begin{tikzpicture}[xscale=0.59, yscale=0.7, every label/.style={inner sep=1pt}]
    \node[draw, circle, fill=black, inner sep=2pt, label={-90:$w$}] (W) at (0, 0) {};
    \foreach \val [count=\i from 1] in {-4, 0, 4} {
        \ifthenelse{\i = 1}{\def\x{180}}{\def\x{0}};
        \node[draw, circle, fill=black, inner sep=2pt, label={\x:$u_{\i}$}] (U\i) at (\val, 1) {};
        \node[draw, circle, fill=black, inner sep=2pt, label={135:$a_{\i}$}] (A\i) at (\val - 1.3, 2) {};
        \node[draw, circle, fill=black, inner sep=2pt, label={-45:$b_{\i}$}] (B\i) at (\val + 1.3, 2) {};
        \node[draw, circle, fill=black, inner sep=2pt, label={0:$d_{\i}$}] (D\i) at (\val, 3) {};
        \node[draw, circle, fill=black, inner sep=2pt, label={90:$v_{\i}$}] (V\i) at (\val, 6) {};
        \draw (U\i) -- (A\i) -- (B\i) -- (D\i) -- (U\i) -- (B\i);
        \draw (A\i) -- (D\i);
        \draw (U\i) -- (W) [line width=3pt];
        \draw (A\i) -- (V\i) -- (B\i) [line width=3pt];
        \draw (D\i) -- (V\i) [line width=3pt];
    }
    \node[draw, circle, fill=black, inner sep=2pt, label={-90:$z$}] (Z) at (10, 0) {};
    \foreach \val [count=\j from 1] in {8, 12} {
        \ifthenelse{\j = 1}{\def\x{180}}{\def\x{0}};
        \node[draw, circle, fill=black, inner sep=2pt, label={\x:$u'_{\j}$}] (UP\j) at (\val, 1) {};
        \node[draw, circle, fill=black, inner sep=2pt, label={135:$a'_{\j}$}] (AP\j) at (\val - 1.3, 2) {};
        \node[draw, circle, fill=black, inner sep=2pt, label={-45:$b'_{\j}$}] (BP\j) at (\val + 1.3, 2) {};
        \node[draw, circle, fill=black, inner sep=2pt, label={0:$d'_{\j}$}] (DP\j) at (\val, 3) {};
        \node[draw, circle, fill=black, inner sep=2pt, label={90:$v'_{\j}$}] (VP\j) at (\val, 6) {};
        \draw (UP\j) -- (AP\j) -- (BP\j) -- (DP\j) -- (UP\j) -- (BP\j);
        \draw (AP\j) -- (DP\j);
        \draw (UP\j) -- (Z) [line width=3pt];
        \draw (AP\j) -- (VP\j) -- (BP\j) [line width=3pt];
        \draw (DP\j) -- (VP\j) [line width=3pt];
    }
    \draw (W) -- (Z);
    \draw (VP1) -- (VP2);
\end{tikzpicture}
\caption{$R$-gadget}
\label{fig:GalaxyRGadget}
\end{figure}

\begin{lemma}
    \label[lemma]{lem:SG2}
    In every $2$-backbone $5$-coloring of an $R$-gadget, $c(z)=3$ and $c(v_1)=c(v_2)=c(v_3)\in\{1,5\}$. Both common values extend to a coloring of the gadget.
\end{lemma}

\begin{proof}
    Each of $v_i,v'_j$ has three pairwise adjacent backbone neighbors, so its color is $1$ or $5$. If its color is $1$, those neighbors use $3,4,5$, leaving its respective vertex $u_i$ or $u'_j$ in $\{1,2\}$. If its color is $5$, the neighbors use $1,2,3$ and the respective vertex lies in $\{4,5\}$.

    The edge $v'_1v'_2$ forces opposite terminal colors. Thus one of $u'_1,u'_2$ is in $\{1,2\}$ and the other is in $\{4,5\}$. Their backbone edges to $z$ imply $c(z)\ge3$ and $c(z)\le3$, respectively, so $c(z)=3$. If two of $v_1,v_2,v_3$ had different colors, the same argument applied to their vertices $u_i$ would force $c(w)=3$, contradicting the edge $wz$.

    For either common value $\beta\in\{1,5\}$, set $c(v_i)=c(u_i)=\beta$, $c(w)=6-\beta$, and assign $a_i,b_i,d_i$ the colors $3,4,5$ when $\beta=1$, or $1,2,3$ when $\beta=5$. Set $c(z)=3$, $c(v'_1)=c(u'_1)=1$, $c(v'_2)=c(u'_2)=5$, and assign the two primed triangles the colors $3,4,5$ and $1,2,3$, respectively. This satisfies all edge constraints.
\end{proof}

\begin{theorem}
    \label{thm:galaxy-lambda+3}
    Let $G$ be a graph with $\Delta(G) \le 4$ and $S$ be a spanning galaxy of $G$. Then for every $\lambda \ge 2$ deciding whether $BBC_\lambda(G, S) \le \lambda + 3$ is $\NP$-complete.
\end{theorem}

\begin{proof}
    We proceed with a reduction from an $\NP$-complete problem \texttt{NAE-4-SAT}, in which for $n$ variables $x_1$, \ldots, $x_n$ and $m$ $4$-CNF clauses $C_1$, \ldots, $C_m$ we ask whether there exists an assignment such that in all clauses there is at least one true literal and at least one false literal. Again, for simplicity, we will use the version of \texttt{NAE-4-SAT} in which all literals appear without negation \cite{schaefer1978complexity}. Discard unused variables and assume $m\ge1$. First suppose $\lambda\ge3$.
    
    For each variable $x_i$ appearing in the formula exactly $k_i$ times we build a gadget $X_i$ by introducing $2 k_i$ $W$-gadgets and connecting $v_{3, j, i}$ with $v_{2, j + 1, i}$ using non-backbone edges to form a cycle over them (identifying $j = 1$ with $j = 2 k_i + 1$) -- where the latter two indices denote the index $j = 1, \ldots, 2 k_i$ of a $W$-gadget and the index of a variable $x_i$, respectively.

    For each clause $C_j$ we create a gadget $K_4$ with all non-backbone edges. Then we connect using backbone edges its vertices representing the literals $x_i$ appearing in this clause to their respective vertices $v_{1, 2 t - 1, i}$ from gadgets $X_i$. Here $t$ denotes that the literal $x_i$ appears in $C_j$ for the $t$-th time, counting from the beginning of the whole formula.

    If the input formula is satisfiable with an assignment $\phi$, set $h_i=(\lambda + 2)\phi(x_i)+1$ and color all three vertices $v_{j,2t-1,i}$ with $h_i$, and all $v_{j,2t,i}$ with $\lambda+4-h_i$. This satisfies the non-backbone links around each variable gadget. Moreover, since each clause contains at least one true and one false literal, its respective $K_4$ has at least one neighbor of color $1$ and one of color $\lambda + 3$. Assign distinct colors from $\{\lambda + 1, \lambda + 2, \lambda + 3\}$ to its vertices whose backbone neighbor has color $1$, and distinct colors from $\{1, 2, 3\}$ to the others. Each group has at most three vertices and the palettes are disjoint, so this colors the clique and satisfies its backbone edges. Finally, it is sufficient to complete the coloring of the whole graph by finding suitable colorings of $W$-gadgets, e.g. if one has $c(v_1) = c(v_2) = c(v_3) = 1$, then one can easily find a suitable coloring, for example using $c(u_1) = c(u_2) = c(u_3) = 1$ and $c(w) = \lambda + 1$ (and colors $\lambda + 1$, $\lambda + 2$, $\lambda + 3$ for the unnamed vertices in cliques) -- and symmetrically we can find one as well in the case $c(v_1) = c(v_2) = c(v_3) = \lambda + 3$.

    There are $2\sum_i k_i = 8m$ copies of $W$ and $m$ clause gadgets, so the construction has size $O(n + m)$. Each used vertex $v_1$ gains one backbone leaf, enlarging its star to $K_{1,4}$; the links between copies of $W$ are non-backbone edges. Thus the backbone remains a galaxy, and all degrees are at most $4$.

    To prove the other direction it is sufficient to note that in a valid $\lambda$-backbone $(\lambda + 3)$-coloring any $K_4$ representing a clause has to contain at least one vertex with a color from the set $\{1, 2, 3\}$ and one with a color from $\{\lambda + 1, \lambda + 2, \lambda + 3\}$. This way, the neighborhood of each such $K_4$ also has to contain at least one vertex with a color from the set $\{1, 2, 3\}$ and one with a color from $\{\lambda + 1, \lambda + 2, \lambda + 3\}$ -- and by \Cref{lem:SG} they have to be equal to either $1$ or $\lambda + 3$. The connections between $W$-gadgets in each $X_i$ ensure the equality of all $c(v_{j, 2 t - 1, i})$ for any fixed $i$, thus we may infer $\phi(x_i) = 1$ when $c(v_{j, 2 t - 1, i}) = \lambda + 3$ and $\phi(x_i) = 0$ otherwise. This assignment gives each clause a true and a false literal.

    For $\lambda=2$, use the same construction with every $W$-gadget replaced by an $R$-gadget. By \Cref{lem:SG2}, the common terminal colors in consecutive copies alternate between $1$ and $5$, so the odd-indexed copies again encode a consistent truth assignment. A clause clique cannot be colored when all four terminals have the same color, since only three colors would be available. Conversely, let $r\in\{1,2,3\}$ be the number of its terminals colored $1$. For $r=1$, give their clause neighbor color $4$ and the other three vertices colors $1,2,3$; for $r=2$, use $4,5$ and $1,2$ on the two groups; for $r=3$, use $3,4,5$ and $1$, respectively. These assignments color the clique properly and satisfy all four backbone edges. Together with the gadget colorings in \Cref{lem:SG2}, this proves both directions of the reduction. Each external edge uses a distinct terminal of degree $3$, and clause edges only add leaves to the corresponding stars. Thus the backbone is a spanning galaxy without isolated vertices, $\Delta(G)\le4$, and the construction has $O(m)$ vertices.
\end{proof}

\section{Path backbone}

For Hamiltonian path backbones, we prove the upper bound $\lambda+4$ for every $\lambda\ge2$. For $\lambda\ge3$, this improves the bound from \Cref{thm:connected-lambda+5} by one. The proof uses the cycle-plus-triangles theorem. We then prove $\NP$-completeness at $\lambda+3$ for $\lambda\ge3$.

\subsection{The easy case}

We use the following cycle-plus-triangles theorem, which was posed as a conjecture by Paul Erd\H{o}s and has an earlier number-theoretic formulation due to Issai Schur \cite{fellows1990transversals}:
\begin{theorem}[\cite{fleischner1992solution}]
    \label{lem:3-split}
    Let $n$ be a positive integer, and let $G$ be a loopless $4$-regular multigraph on $3 n$ vertices. Assume that $G$ has a decomposition into a Hamiltonian circuit and $n$ pairwise vertex-disjoint triangles. Then $\chi(G) = 3$.
\end{theorem}

This theorem yields the following upper bound:
\begin{theorem}
    \label{thm:path-lambda+4}
    Let $G$ be a graph with $\Delta(G) \le 4$ and let $P$ be a Hamiltonian path in $G$.
    Then for $\lambda \ge 2$ it holds that $BBC_\lambda(G, P) \le \lambda + 4$. A suitable coloring can be found in $O^\text{*}\!(1.3217^n)$ time.
\end{theorem}

\begin{proof}
    First, let us number the vertices $v_1$, \ldots, $v_n$ according to their order of appearance in $P$. Let also $A$ and $B$ denote the sets of vertices with odd and even indices, respectively. Our whole idea is to find a coloring $c$ of $G$ such that the vertices from $A$ get colors $\{1, 2, 3\}$ and the vertices from $B$ get $\{\lambda + 2, \lambda + 3, \lambda + 4\}$ so that all the backbone edge constraints are satisfied. It suffices to ensure that $C_3 \cup C_{\lambda + 2}$ (recall from the previous section that $C_i$ is a set of vertices colored with $i$) is an independent set in $G$ and that non-backbone edges in $G \setminus (C_3 \cup C_{\lambda + 2})$ form a bipartite graph.

    Let us see first that the theorem holds for $n \le 2$.
    When $n \ge 3$, then we have three possible cases: either $(a)$ $v_1$ and $v_n$ are adjacent in $G \setminus P$, or $(b)$ they are not adjacent, but they are in the same connected component of $G \setminus P$ or $(c)$ they belong to different connected components of $G \setminus P$. We define a graph $P'$ and a set of vertices $X$ separately for each case:
    \begin{description}
        \item[case $(a)$:] let us denote by $P'$ a Hamiltonian cycle consisting of $P$ together with an edge $v_1v_n$ and let $X = \emptyset$.
        Note that the resulting coloring would not take into account this additional edge as a backbone one -- it is just needed to map the problem more easily into a known one.
        \item[case $(b)$:] let $X$ be a set of internal vertices (that is, excluding the endpoints $v_1$ and $v_n$) on any path between $v_1$ and $v_n$ in $G \setminus P$. Let also $P' = P \setminus X$.
        \item[case $(c)$:] observe that $\deg_{G \setminus P}(v_1) \le 3$, $\deg_{G \setminus P}(v_n) \le 3$, and $\deg_{G \setminus P}(u) \le 2$ for all other vertices $u$.
        Thus, each component containing $v_1$ or $v_n$ consists of at most three paths meeting at that endpoint, or of a cycle through it and at most one attached path. Let $X$ contain all vertices other than $v_1,v_n$ on the maximal paths from these endpoints to degree-one vertices in $G\setminus P$, excluding vertices of any cycle through an endpoint. Again, let $P' = P \setminus X$.
    \end{description}
    In any case, we know that $P'$ is either a Hamiltonian cycle or a spanning collection of paths in $G' = G \setminus X$. Moreover, by construction, $\deg_{G' \setminus P'}(v_1) \le 2$, $\deg_{G' \setminus P'}(v_n) \le 2$, and $\deg_{G' \setminus P'}(u) \le \deg_{G \setminus P}(u) \le 2$ for any other vertex $u$ so $\Delta(G' \setminus P') \le 2$.

    Thus, $G' \setminus P'$ is a collection of paths and cycles.
    If there are no odd cycles, take $V(G^*) = C^*_3 = \emptyset$ and skip the following auxiliary construction. Otherwise, for every odd cycle $O_i$ in $G' \setminus P'$ let us pick three arbitrary vertices $u_{i, 1}$, $u_{i, 2}$, $u_{i, 3}$ with the only condition that if $O_i$ contains either $v_1$ or $v_n$, then they have to be among those vertices.

    Write $U=\{u_{i,j}\}$. If $P'[U]$ is already a Hamiltonian cycle, set $P^*=P'[U]$. Otherwise, $P'[U]$ is a collection of paths (including isolated vertices); concatenate these paths in any order and close them into a Hamiltonian cycle $P^*$ on $U$. This is possible since $|U|\ge3$, and preserves every edge of $P'[U]$.
    Obtain $G^*$ from $P^*$ by adding the three triangle edges $u_{i,1}u_{i,2}$, $u_{i,1}u_{i,3}$ and $u_{i,2}u_{i,3}$ for every $O_i$, retaining parallel edges if necessary.

    This way, $G^*$ is a $4$-regular multigraph consisting only of a Hamiltonian cycle $P^*$ and triangles, so we can apply \Cref{lem:3-split} to find its $3$-coloring $c^*$ with $|C^*_1| = |C^*_2| = |C^*_3|$.
    By permuting colors, assume that if $v_1 \in V(G^*)$, then $c^*(v_1) = 3$. Otherwise, make color $3$ absent from $v_2$ and $v_n$ whenever these vertices belong to $V(G^*)$; at most two colors must be avoided.

    Put $Z=C^*_3\cup\{v_1\}$. This set is independent in $P'$: edges between selected vertices are preserved in $P^*$, and the choice of color $3$ excludes the possible neighbors $v_2,v_n$ of an additionally included $v_1$. It is also independent in $G'\setminus P'$, since $C^*_3$ contains one vertex from each odd-cycle component and no others. If $v_1$ lies on such a cycle, it is that cycle's selected vertex; otherwise its component contains no vertex of $C^*_3$.

    Let us set $c(v) = 3$ for all $v \in C^*_3 \cap A$ and $c(v) = \lambda + 2$ for all $v \in C^*_3 \cap B$.
    If $v_1 \notin V(G^*)$, then we additionally set $c(v_1) = 3$ (since $v_1 \in A$) -- thus, in either case we are certain that $c(v_1) = 3$.

    Now we want to show how to assign colors to $G'' = G' \setminus (C^*_3 \cup \{v_1\})$. First, we note that $G''$ induces the backbone $P'' = P'[V(G'')] = P[V(G'')]$ in a way such that $G'' \setminus P''$ is bipartite.
    To prove this, it is sufficient to recall that $G' \setminus P'$ was a collection of paths and cycles, and that $C^*_3$ contained exactly one vertex from each triangle in $G^* \setminus P^*$ -- thus exactly one vertex from each odd cycle from $G' \setminus P'$. Since $C^*_3$ hits every odd cycle in $G' \setminus P'$, removing it leaves a bipartite graph, and a possible further removal of $v_1$ does not change this property.

    Since $G'' \setminus P''$ is bipartite, we can find its $2$-coloring $c''$ and then assign $c(v) = c''(v)$ for all $v \in V(G'') \cap A$ and $c(v) = \lambda + 5 - c''(v)$ for all $v \in V(G'') \cap B$. This way we obtain a coloring of $G'$ such that:
    \begin{itemize}
        \item all backbone edge conditions from $P[V(G')]$ are met by the assignment to $A$ and $B$ and the independence of $Z$ in $P'$,
        \item all non-backbone edge conditions within $G''$ hold: endpoints in the same set $A$ or $B$ retain distinct colors from $c''$, while endpoints in different sets use disjoint palettes,
        \item all non-backbone edge conditions between $V(G'')$ and $Z$ hold because these sets use disjoint palettes $\{1, 2, \lambda + 3, \lambda + 4\}$ and $\{3, \lambda + 2\}$; there are no edges of $G'\setminus P'$ within $Z$.
    \end{itemize} 
    In case $(a)$, the additional non-backbone edge $v_1v_n$ is also properly colored: $c(v_1)=3$ and $v_n\notin Z$. It need not satisfy a backbone difference.

    Finally, we are left with the only uncolored vertices in $X$. In case $(a)$ it is trivial since $X$ is empty.

    If $X \cup \{v_1, v_n\}$ induces in $G \setminus P$ a single path between $v_1$ and $v_n$ (case $(b)$), then we proceed sequentially along the path starting from $v_n$ and note that for every $u \in X$ its membership in $A$ or $B$ ensures that we pick either color from the set $\{1, 2\}$ or $\{\lambda + 3, \lambda + 4\}$ so that both backbone edge conditions are met. Moreover, since $u$ is adjacent in $G \setminus P$ to at most one already colored vertex in $X \cup \{v_n\}$, then there is always a color satisfying a non-backbone edge to the previous vertex along the path. And clearly the last vertex on this path is also adjacent to $v_1$, but since $c(v_1) = 3$ this condition is met as well.

    Otherwise $X$ was determined in the case $(c)$ above, and we know that $X \cup \{v_1, v_n\}$ induces a collection of paths in $G \setminus P$, each starting from $v_1$ or $v_n$ and otherwise vertex-disjoint. Color each path away from its already colored endpoint, using the same two-color choices as above. This again gives a valid backbone coloring of the whole graph.

    All steps take linear time except finding $c^*$. The existence proof in \cite{fleischner1992solution} uses the algebraic coloring method of Alon and Tarsi \cite{alon1992colorings}, related to the Combinatorial Nullstellensatz \cite{alon1999combinatorial}, and does not supply a polynomial-time coloring algorithm. Applying the general $3$-coloring algorithm of \cite{meijer2023} to the underlying simple graph of $G^*$ gives the claimed $O^\text{*}\!(1.3217^n)$ bound, since $|V(G^*)|\le n$.
\end{proof}

\subsection{The hard case}

In order to show the hardness of $\lambda$-backbone $(\lambda + 3)$-coloring of graphs with Hamiltonian path backbones we need to introduce several auxiliary graphs necessary for reductions. 
First, for some $t \ge 2$ let us define graphs $P$, $C$, $M$, $G$ with $V(P) = V(C) = V(M) = V(G) = \{v_1, \ldots, v_{8 t}\}$ and:
\begin{itemize}
    \item $E(P) = \{v_iv_{i+1}\colon i = 1, \ldots, 8 t - 1\}$,
    \item $E(C) = \{v_{4 i - 1}v_{4 i + 3}\colon i = 1, \ldots, 2 t - 1\} \cup \{v_{8 t - 1}v_3\}$,
    \item $E(M) = \{v_{4 i - 2}v_{4 i}\colon i = 1, \ldots, 2 t\}$,
    \item $E(G) = E(P) \cup E(C) \cup E(M)$
\end{itemize}
Let us also define the graph $G$ with backbone $P$ as a \emph{PCM-gadget}. An example of such a gadget is presented in \Cref{fig:gadgetPCM}.

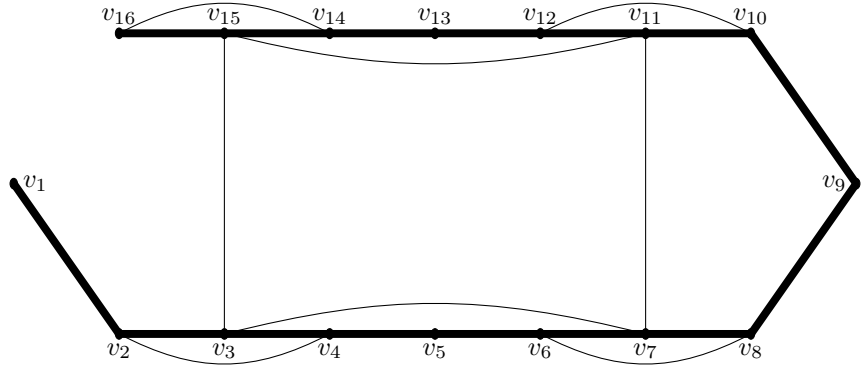
\begin{figure}[H]
  \centering
  \begin{tikzpicture}[xscale=0.7]
      \foreach \val [count=\i from 10] in {0, 2, 4, 6, 8, 10, 12} {
          \coordinate[label=90:$v_{\i}$] (v\i) at (14 - \val, 4);
          \filldraw (v\i) circle (2pt);
      }
      \foreach \i \val \angle in {1/0/0, 9/16/180} {
          \coordinate[label=\angle:$v_{\i}$] (v\i) at (\val, 2);
          \filldraw (v\i) circle (2pt);
      }
      \foreach \val [count=\i from 2] in {2, 4, 6, 8, 10, 12, 14} {
          \coordinate[label=270:$v_{\i}$] (v\i) at (\val, 0);
          \filldraw (v\i) circle (2pt);
      }
      \foreach \from \to in {1/2, 2/3, 3/4, 4/5, 5/6, 6/7, 7/8} {
          \draw (v\from) -- (v\to) [line width=3pt];
      }
      \draw (v8) -- (v9) [line width=3pt];
      \foreach \from \to in {9/10, 10/11, 11/12, 12/13, 13/14, 14/15, 15/16} {
          \draw (v\from) -- (v\to) [line width=3pt];
      }
      \foreach \u \v in {2/4, 6/8} {
          \draw (v\u) to[out=340, in=200] (v\v);
      }
      \foreach \u \v in {10/12, 14/16} {
          \draw (v\u) to[out=160, in=20] (v\v);
      }
      \foreach \u \v in {3/7} {
          \draw (v\u) to[out=10, in=170] (v\v);
      }
      \draw (v7) to (v11);
      \foreach \u \v in {11/15} {
          \draw (v\u) to[out=190, in=350] (v\v);
      }
      \draw (v15) to (v3);
  \end{tikzpicture}
  \caption{PCM-gadget for $t = 2$ (on $8 t = 16$ vertices) with backbone edges in bold.}
  \label{fig:gadgetPCM}
\end{figure}

Throughout this section we will refer to $v_1$ and $v_{8 t}$ as the \emph{beginning} and the \emph{end} of a PCM-gadget, respectively.

\begin{lemma}
    \label{lem:variable_gadget}
    For $\lambda \ge 3$, every $\lambda$-backbone $(\lambda + 3)$-coloring $c$ of a PCM-gadget such that $c(v_1) \in \{1, 2, 3\}$ has the following properties:
    \begin{enumerate}
        \item all $v_{8 i - 5}$ ($i = 1, 2, \ldots, t$) have the same color $a$,
        \item all $v_{8 i - 1}$ ($i = 1, 2, \ldots, t$)  have the same color $b$,
        \item $\{a, b\} = \{1, 2\}$.
    \end{enumerate}
\end{lemma}

\begin{proof}
    The palettes $\{1,2,3\}$ and $\{\lambda+1,\lambda+2,\lambda+3\}$ alternate along $P$, starting with the first palette at $v_1$. If $c(v_{4j-1})=3$, both backbone neighbors $v_{4j-2},v_{4j}$ would have color $\lambda+3$, contradicting their adjacency in $M$. Hence $c(v_{4j-1})\in\{1,2\}$ for every $j=1,\ldots,2t$.

    To finish the proof it is sufficient to see that vertices $v_{4 i - 1}$ ($i = 1, 2, \ldots, 2t$) induce an even cycle in the PCM-gadget (i.e. $C$ itself). This guarantees the equality of colors of all vertices within both sets, $\{v_{8 i - 5}\colon i = 1, \ldots, t\}$ and $\{v_{8 i - 1}\colon i = 1, \ldots, t\}$.
\end{proof}

\begin{figure}[H]
\centering
\begin{tikzpicture}[xscale=0.82]
    \foreach \val [count=\i] in {0, 2, 4, 6, 8, 10, 12, 14} {
        \coordinate[label={[label distance=1mm]270:$u_{\i}$}] (v\i) at (\val, 0);
        \filldraw (v\i) circle (2pt);
    }
    \foreach \from \to in {1/2, 2/3, 3/4, 4/5, 5/6, 6/7, 7/8} {
        \draw (v\from) -- (v\to) [line width=3pt];
    }
    \foreach \from \to in {3/5, 5/7, 3/7} {
        \draw (v\from) to[out=45, in=135] (v\to);
    }
    \foreach \from \to in {2/4, 4/6} {
        \draw (v\from) to[out=315, in=225] (v\to);
    }
\end{tikzpicture}
\caption{Force gadget with backbone edges in bold.}
\label{fig:gadgetF}
\end{figure}
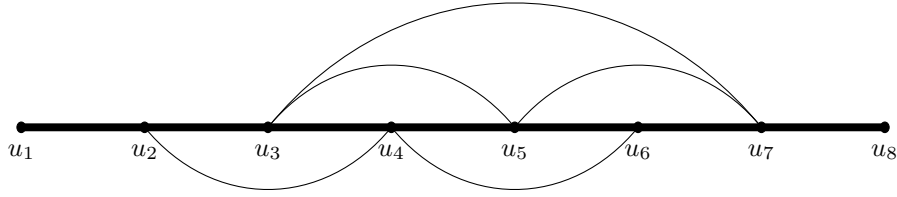

Let a \emph{force gadget} be a graph $G$ with backbone $P$ (presented in \Cref{fig:gadgetF}) such that $V(G) = V(P) = \{u_1, \ldots, u_8\}$, $E(P) = \{u_i u_{i + 1}\colon i = 1, 2, \ldots, 7\}$, $E(G) = E(P) \cup \{u_2u_4, u_4u_6, u_3u_5, u_3u_7, u_5u_7\}$.
As before, we will refer to $u_1$ and $u_8$ as the \emph{beginning} and the \emph{end} of a force gadget, respectively.

\begin{lemma}
    \label{lem:force_gadget}
    For $\lambda \ge 3$, every $\lambda$-backbone $(\lambda + 3)$-coloring $c$ of a force gadget such that $c(u_1) \in \{1, 2, 3\}$ has $c(u_8) = \lambda + 3$.
\end{lemma}

\begin{proof}
    First, we notice that $c(u_1) \in \{1, 2, 3\}$ ensures that the odd vertices on $P$ get colors $\{1, 2, 3\}$ and the even ones $\{\lambda + 1, \lambda + 2, \lambda + 3\}$.

    Next, it cannot be the case that $c(u_3) = 3$, because this would force $c(u_2) = c(u_4) = \lambda + 3$ -- impossible since $u_2u_4 \in E(G)$. By the same reasoning it holds that $c(u_5) \neq 3$.
    However, since $\{u_3, u_5, u_7\}$ forms a triangle in $G$, it follows that $c(u_7) = 3$, and therefore $u_8$ as its neighbor in $P$ must have been assigned color $\lambda + 3$ in $c$.
\end{proof}

\begin{theorem}
    \label{thm:path-lambda+3}
    Let $G$ be a graph with $\Delta(G) \le 4$ and $P$ be a Hamiltonian path in $G$. Then for every $\lambda \ge 3$ deciding whether $BBC_\lambda(G, P) \le \lambda + 3$ is $\NP$-complete.
\end{theorem}

\begin{proof}
    We proceed with a reduction from \texttt{3-SAT}.
    We use the standard restriction to clauses containing exactly three distinct variables. Delete unused variables and duplicate each clause once; this preserves satisfiability and ensures that every remaining variable occurs at least twice. Let the resulting formula have $m$ clauses $C_1, \ldots, C_m$ and $n$ variables $x_1, \ldots, x_n$, with $t_i \ge 2$ occurrences of $x_i$.
    
    First, we build $n$ variable gadgets such that the $i$-th one has exactly $8 t_i$ vertices.
    We also build $m$ force gadgets. Then, we connect ends and beginnings of gadgets with backbone edges to obtain one long Hamiltonian backbone path $P$ in such a resulting graph. Since $\sum_i t_i = 3m$, the construction has $32m$ vertices and $O(m)$ edges and can be performed in polynomial time.

    Finally, let us add edges representing clauses in the following way (see the example in \Cref{fig:GK3620}): suppose that $C_i$ contains variables $x_j$, $x_k$, $x_l$ (possibly negated) and this is their $s_j$-th, $s_k$-th, and $s_l$-th appearance in the whole formula, respectively. Clearly, $s_j \in \{1, \ldots, m\}$ etc.
    Now choose vertices $w^j_1 = v^j_{8 s_j - 2}$, $w^j_2 = v^j_{8 s_j}$ (with the superscript $j$ denoting the $j$-th variable gadget) if $x_j$ is negated in $C_i$, otherwise $w^j_1 = v^j_{8 s_j - 6}$, $w^j_2 = v^j_{8 s_j - 4}$, and analogously for other variables in $C_i$, and we add the following four edges to our graph $G$: $u^i_8w^j_1$, $w^j_2w^k_1$, $w^k_2w^l_1$, and $w^l_2u^i_8$ (here $u^i_8$ is the end of the $i$-th force gadget).
    Observe that regardless of the choice of $w$'s, adding these edges causes a cycle on exactly these seven vertices to appear in $G \setminus P$ since $w^j_1w^j_2, w^k_1w^k_2, w^l_1w^l_2 \in E(G)$ -- in particular, they are the edges from $E(M)$.

\begin{figure}[htpb]
    \centering
    \begin{tikzpicture}[scale=0.55]

        \foreach \val [count=\i] in {0, 2, 4, 6, 8, 10, 12, 14} {
            \coordinate[label={[label distance=1mm]270:$u_{\i}$}] (p\i) at (-3 + \val, 4);
            \filldraw (p\i) circle (2pt);
        }
        \foreach \from \to in {1/2, 2/3, 3/4, 4/5, 5/6, 6/7, 7/8} {
            \draw (p\from) -- (p\to) [line width=3pt];
        }
        \foreach \from \to in {3/5, 5/7, 3/7} {
            \draw (p\from) to[out=45, in=135] (p\to);
        }
        \foreach \from \to in {2/4, 4/6} {
            \draw (p\from) to[out=315, in=225] (p\to);
        }
        
        \foreach \val [count=\i from 10] in {0, 1, 2, 3, 4, 5, 6} {
            \coordinate[label={[label distance=1mm]270:$v_{\i}$}] (v\i) at (3 - \val, 1);
            \filldraw (v\i) circle (2pt);

            \coordinate[label=270:$v_{\i}'$] (u\i) at (6 - \val, -3);
            \filldraw (u\i) circle (2pt);

            \coordinate[label=270:$v_{\i}''$] (w\i) at (15 - \val, 1);
            \filldraw (w\i) circle (2pt);
        }

        \foreach \i \val in {1/-2, 9/6} {
            \coordinate[label={[label distance=1mm]270:$v_{\i}$}] (v\i) at (-2 + \val, 0);
            \filldraw (v\i) circle (2pt);

            \coordinate[label={[label distance=0.5mm]270:$v_{\i}'$}] (u\i) at (1 + \val, -4);
            \filldraw (u\i) circle (2pt);

            \coordinate[label={[label distance=0.5mm]270:$v_{\i}''$}] (w\i) at (10 + \val, 0);
            \filldraw (w\i) circle (2pt);
        }

        \foreach \val [count=\i from 2] in {0, 1, 2, 3, 4, 5, 6} {
            \coordinate[label={[label distance=1mm]270:$v_{\i}$}] (v\i) at (-3 + \val, -1);
            \filldraw (v\i) circle (2pt);

            \coordinate[label=270:$v_{\i}'$] (u\i) at (0 + \val, -5);
            \filldraw (u\i) circle (2pt);

            \coordinate[label=270:$v_{\i}''$] (w\i) at (9 + \val, -1);
            \filldraw (w\i) circle (2pt);
        }

        \foreach \from \to in {1/2, 2/3, 3/4, 4/5, 5/6, 6/7, 7/8} {
            \draw (v\from) -- (v\to) [line width=3pt];
            \draw (u\from) -- (u\to) [line width=3pt];
            \draw (w\from) -- (w\to) [line width=3pt];
        }

        \draw (v8) -- (v9) [line width=3pt];
        \draw (u8) -- (u9) [line width=3pt];
        \draw (w8) -- (w9) [line width=3pt];

        \foreach \from \to in {9/10, 10/11, 11/12, 12/13, 13/14, 14/15, 15/16} {
            \draw (v\from) -- (v\to) [line width=3pt];
            \draw (u\from) -- (u\to) [line width=3pt];
            \draw (w\from) -- (w\to) [line width=3pt];
        }

        \foreach \u \v in {2/4, 6/8} {
            \draw (v\u) to[out=315, in=225, looseness=0.5] (v\v);
            \draw (u\u) to[out=315, in=225, looseness=0.5] (u\v);
            \draw (w\u) to[out=315, in=225, looseness=0.5] (w\v);
        }

        \foreach \u \v in {10/12, 14/16} {
            \draw (v\u) to[out=135, in=45] (v\v);
            \draw (u\u) to[out=135, in=45] (u\v);
            \draw (w\u) to[out=135, in=45] (w\v);
        }

        \foreach \u \v in {3/7} {
            \draw (v\u) to[out=45, in=135, looseness=0.5] (v\v);
            \draw (u\u) to[out=45, in=135, looseness=0.5] (u\v);
            \draw (w\u) to[out=45, in=135, looseness=0.5] (w\v);
        }

        \draw (v7) to[out=45, in=315, looseness=1.75] (v11);
        \draw (u7) to[out=45, in=315, looseness=1.75] (u11);
        \draw (w7) to[out=45, in=315, looseness=1.75] (w11);

        \foreach \u \v in {11/15} {
            \draw (v\u) to[out=225, in=315, looseness=1.75] (v\v);
            \draw (u\u) to[out=225, in=315, looseness=1.75] (u\v);
            \draw (w\u) to[out=225, in=315, looseness=1.75] (w\v);
        }

        \draw (v15) to[out=225, in=135, looseness=1.75] (v3);
        \draw (u15) to[out=225, in=135, looseness=1.75] (u3);
        \draw (w15) to[out=225, in=135, looseness=1.75] (w3);

        \draw (p8) to[out=225, in=90] (v12);
        \draw (v10) to[out=30, in=90, looseness=2] (u12);
        \draw (u10) to[out=60, in=150, looseness=2] (w12);
        \draw (w10) to[out=90, in=315] (p8);
        
    \end{tikzpicture}
    \caption{A clause gadget example for a clause $(x \lor x' \lor x'')$ with its respective $C_7$ cycle on vertices $\{u_8, v_{12}, v_{10}, v'_{12}, v'_{10}, v''_{12}, v''_{10}\}$. Backbone edges are shown in bold.}
    \label{fig:GK3620}
\end{figure}
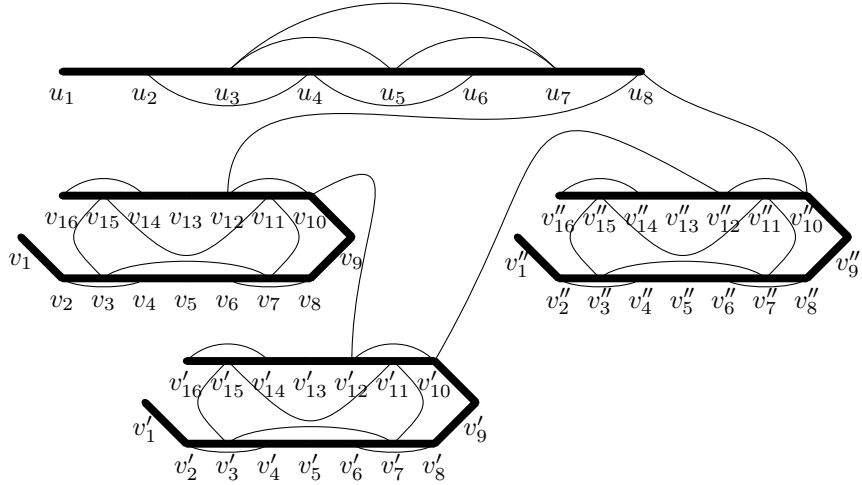

    It is sufficient to prove $(a)$ that the resulting graph $G$ with backbone $P$ has $\Delta(G) \le 4$ and that $(b)$ the input formula is satisfiable if and only if $BBC_\lambda(G, P) \le \lambda + 3$.

    For $(a)$, joining gadgets increases each beginning and end degree by at most $1$. Every even-indexed variable-gadget vertex gets at most one clause edge; its original degree is $3$, except at the end, where it is $2$. Each force-gadget end $u_8$ has original degree $1$ and gets two clause edges and at most one joining edge. Beginnings have original degree $1$, and all other degrees remain at most $4$. Thus the resulting graph has maximum degree at most $4$ and a Hamiltonian backbone path.

    Next, we proceed with a proof of the \textbf{if} part of the property $(b)$. Assume that there exists a valid $\lambda$-backbone $(\lambda + 3)$-coloring $c$ of $G$ with backbone $P$.
    Since $\lambda \ge 3$, the disjoint palettes $\{1,2,3\}$ and $\{\lambda + 1,\lambda + 2,\lambda + 3\}$ alternate along $P$. Every gadget has an even number of vertices, so all beginnings belong to the same palette. Reflecting all colors by $a \mapsto \lambda + 4 - a$ if necessary, we may assume that $c(u^j_1), c(v^i_1) \in \{1,2,3\}$ for $j = 1,\ldots,m$ and $i = 1,\ldots,n$. Thus every odd-indexed vertex within a gadget uses the first palette and every even-indexed vertex uses the second.
    Let us next define the assignment $\phi$ for the input 3-SAT formula such that for every $i = 1, \ldots, n$ we set $\phi(x_i) = 1$ if $c(v^i_3) = 1$, and $\phi(x_i) = 0$ otherwise.
    
    Suppose a clause $C_i$ has all literals false. By \Cref{lem:variable_gadget}, the common backbone neighbor of each selected pair $w_1,w_2$ then has color $2$, so all six $w$ vertices use only $\lambda+2,\lambda+3$. Also $c(u^i_8)=\lambda+3$ by \Cref{lem:force_gadget}. This would properly color the clause's seven-cycle with only two colors, a contradiction. Hence $\phi$ satisfies the formula.

    Now we proceed with a proof of the \textbf{only if} part of the property $(b)$. Suppose that we have a satisfying assignment $\phi$. First, for every force gadget we set $c(u_1) = c(u_3) = 1$, $c(u_5) = 2$, $c(u_7) = 3$, $c(u_4) = \lambda + 2$, $c(u_2) = c(u_6) = c(u_8) = \lambda + 3$. One can verify that such a coloring does not violate any edge condition within force gadgets. Second, we set $c(v^i_{8 j - 5}) = 1$ and $c(v^i_{8 j - 1}) = 2$ for $j = 1, \ldots, t_i$ for every $i$ such that $\phi(x_i) = 1$. Otherwise, we assign $c(v^i_{8 j - 5}) = 2$ and $c(v^i_{8 j - 1}) = 1$.

    Next, for every clause $C_i$ we color the respective cycle in the following way: if the literal $x$ or $\lnot x$ in the clause is true, then we assign $c(\{w_1, w_2\}) = \{\lambda + 1, \lambda + 2\}$ to its respective vertices, otherwise we set $c(\{w_1, w_2\}) = \{\lambda + 2, \lambda + 3\}$, using notation introduced in previous sections. To specify the orientations, write $A = \lambda + 1$, $B = \lambda + 2$, $C = \lambda + 3$, and traverse the cycle from $u^i_8$, colored $C$. Assign the ordered pair $(A,B)$ to each true literal, $(B,C)$ to each false literal before the first true one, and $(C,B)$ to each remaining false literal. Consecutive colors on the cycle are distinct, including its last color $B$ and the initial color $C$.

    Complete each variable gadget by setting $c(v^i_{4j-3}) = 1$ for $j = 1,\ldots,2t_i$, and assigning $(c(v^i_{4j-2}),c(v^i_{4j})) = (\lambda + 2,\lambda + 3)$ to each as yet uncolored matching edge. The vertices of its cycle $C$ already alternate between $1$ and $2$. Each edge of $M$ has distinct endpoint colors, either by the clause assignment or by this completion. All backbone differences are at least $\lambda$: colors $\lambda + 2$ and $\lambda + 3$ are compatible with both $1$ and $2$, while color $\lambda + 1$ occurs only at a true literal, where both backbone neighbors have color $1$.

    Finally, all gadget beginnings have color $1$ and all gadget ends have color at least $\lambda + 1$. Therefore the backbone edges joining successive gadgets also satisfy the required difference.
\end{proof}

\section{Tree backbone}

For spanning tree backbones, the upper bound $\lambda+5$ follows from \Cref{thm:connected-lambda+5}, while hardness at $\lambda+3$ for $\lambda\ge3$ follows from the path result. We first record the upper bound as a corollary, then extend hardness to $\lambda=2$ and establish hardness at $\lambda+4$ for $\lambda\ge4$. The last result shows that the general upper bound is sharp.

\subsection{The easy case}

Since every spanning tree is connected and bipartite, \Cref{thm:connected-lambda+5} immediately gives the following result:
\begin{corollary}
    \label[corollary]{thm:tree-lambda+5}
    Let $G$ be a graph with $\Delta(G)\le4$ and $T$ be its spanning tree. Then for every fixed $\lambda\ge3$ it holds that $BBC_\lambda(G,T)\le\lambda+5$. A suitable coloring can be found in $O(n)$ time.
\end{corollary}

\subsection{The hard cases}

We first extend the hardness result for path backbones to separation $2$ when the backbone is a spanning tree.
\begin{theorem}
    \label{thm:tree-lambda+3}
    Let $G$ be a graph with $\Delta(G)\le4$ and $T$ be its spanning tree. Then for every $\lambda\ge2$ deciding whether $BBC_\lambda(G,T)\le\lambda+3$ is $\NP$-complete.
\end{theorem}

\begin{proof}
    For $\lambda\ge3$, the result follows immediately from \Cref{thm:path-lambda+3}. For $\lambda=2$, take the graph and galaxy constructed in the proof of \Cref{thm:galaxy-lambda+3}, with $2k_i$ copies of $R$ for each variable $x_i$. We modify its backbone into a spanning tree.

    In each $R$-gadget, make the existing edges $u_1a_1,u_2a_2,u_3a_3,u'_1a'_1,v'_1v'_2,wz$ backbone edges. The first three join the four stars of its $W$-part into one tree. The next two join the three stars of the attached part into another tree, and $wz$ joins the two trees. Every newly designated backbone edge has color difference at least $2$ in the colorings given in \Cref{lem:SG2}.

    For each variable $x_i$, also make the edges $v_{3,j,i}v_{2,j+1,i}$, for $1\le j<2k_i$, backbone edges, leaving the closing edge non-backbone. These edges join its gadget trees in a path and have endpoint colors $1$ and $5$ in every coloring used for completeness. The clause vertices remain backbone leaves, so there is now one backbone tree per variable.

    Write $z_{j,i}$ for the vertex $z$ in the $j$-th copy of $R$ for $x_i$. For $i=1,\ldots,n-1$, add a fresh vertex $h_i$ and backbone edges $z_{2k_i,i}h_i$ and $h_i z_{1,i+1}$. Each such vertex $z_{j,i}$ had degree $3$ and is used at most once, since $2k_i\ge2$. Thus all graph degrees remain at most $4$. The new paths connect the variable trees in a chain, so the backbone is a spanning tree.

    Every satisfying assignment extends to the original galaxy coloring using \Cref{lem:SG2}. The promoted edges satisfy the required differences, and assigning color $1$ to every $h_i$ satisfies its edges because all vertices $z_{j,i}$ have color $3$. Conversely, a coloring of the new instance restricts to a valid coloring of the galaxy instance after deleting the vertices $h_i$ and forgetting the additional backbone requirements. Hence the original formula is satisfiable. The construction has $220m+n-1$ vertices and takes polynomial time. Together with membership in $\NP$, this completes the proof.
\end{proof}

Throughout the remainder of this subsection, let $\lambda \ge 4$.

Although the proof is probably the longest one, it follows the pattern outlined in the hardness proofs above. That is, for the major part we just construct gadgets that either forbid or force certain colors in a $\lambda$-backbone $(\lambda + 4)$-coloring to be used in a reduction.

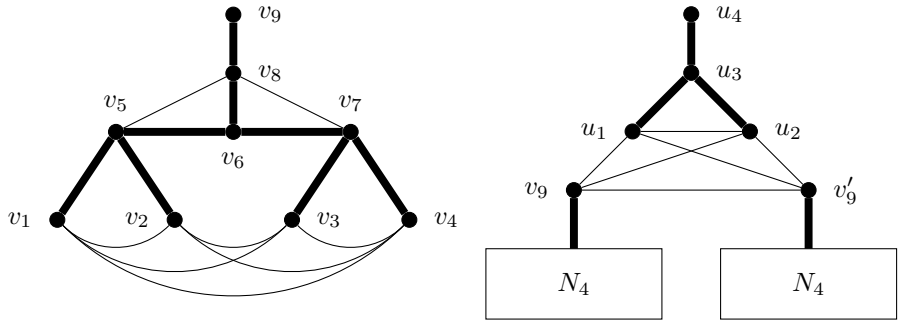
\begin{figure}[H]
\centering
\begin{tikzpicture}[scale=0.78]

    \foreach \val [count=\i from 9] in {3} {
        \node[draw, circle, fill=black, inner sep=2pt, label={[label distance=1mm]000:$v_{\i}$}] (v\i) at (\val, 3) {};
    }

    \foreach \val [count=\i from 8] in {3} {
        \node[draw, circle, fill=black, inner sep=2pt, label={[label distance=1mm]0:$v_{\i}$}] (v\i) at (\val, 2) {};
    }

    \foreach \val [count=\i from 5] in {1, 3, 5} {
        \ifthenelse{\i = 6}{\def\x{270}}{\def\x{090}};
        \node[draw, circle, fill=black, inner sep=2pt, label={[label distance=0.5mm]\x:$v_{\i}$}] (v\i) at (\val, 1) {};
    }
    
    \foreach \val [count=\i] in {0, 2, 4, 6} {
        \ifthenelse{\i < 3}{\def\x{180}}{\def\x{0}};
        \node[draw, circle, fill=black, inner sep=2pt, label={[label distance=1mm]\x:$v_{\i}$}] (v\i) at (\val, -0.5) {};
    }

    \draw (v1) -- (v5) -- (v2) [line width=3pt];
    \draw (v3) -- (v7) -- (v4) [line width=3pt];
    \draw (v5) -- (v6) -- (v7) [line width=3pt];
    \draw (v6) -- (v8) -- (v9) [line width=3pt];

    \draw (v5) -- (v8) -- (v7);

    \foreach \from in {1, 2, 3, 4} {
        \foreach \to in {1, 2, 3, 4} {
            \ifthenelse{\from < \to}{\draw (v\from) to[out=-45, in=-135] (v\to);}{}
        }
    }
\end{tikzpicture}
\ 
\begin{tikzpicture}[scale=0.78]
    \node[draw, circle, fill=black, inner sep=2pt, label={[label distance=1mm]0:$u_4$}] (v22) at (7, 4) {};
    \node[draw, circle, fill=black, inner sep=2pt, label={[label distance=1mm]0:$u_3$}] (v21) at (7, 3) {};
    \node[draw, circle, fill=black, inner sep=2pt, label={[label distance=1mm]180:$u_1$}] (v18) at (6, 2) {};
    \node[draw, circle, fill=black, inner sep=2pt, label={[label distance=1mm]0:$u_2$}] (v19) at (8, 2) {};

    \node[draw, circle, fill=black, inner sep=2pt, label={[label distance=1mm]180:$v_9$}] (v17) at (5, 1) {};
    \node[draw, circle, fill=black, inner sep=2pt, label={[label distance=1mm]0:$v'_9$}] (v20) at (9, 1) {};

    \coordinate (v15) at (5, -0);
    \coordinate (v16) at (9, -0);

    \draw (3.5, -0) rectangle (6.5, -1.2);
    \node at (5, -0.6) {$N_4$};
    \draw (7.5, -0) rectangle (10.5, -1.2);
    \node at (9, -0.6) {$N_4$};

    \draw (v16) -- (v20) [line width=3pt];
    \draw (v15) -- (v17) [line width=3pt];

    \foreach \from in {17, 18, 19, 20} {
        \foreach \to in {17, 18, 19, 20} {
            \ifthenelse{\from < \to}{\draw (v\from) to (v\to);}{}
        }
    }

    \draw (v18) -- (v21) -- (v22) [line width=3pt];
    \draw (v19) -- (v21) [line width=3pt];
\end{tikzpicture}
\caption{$N_4$-gadget (left) and $F_{\lambda + 4}$-gadget (right) with their backbone edges in bold.}
\label{fig:n4-f8-gadget}
\end{figure}

We start by defining the \emph{$N_4$-gadget}, presented in \Cref{fig:n4-f8-gadget}, and proving its important property:
\begin{lemma}
    \label{lem:n4-gadget}
    Let $c$ be a $\lambda$-backbone $(\lambda + 4)$-coloring of the $N_4$-gadget such that $c(v_6) \in \{1, 2, 3, 4\}$. Then, $c(v_9) \in \{1, 2, 3\}$.
\end{lemma}

\begin{proof}
    The backbone edge conditions imply that if $c(v_6) \in \{1, 2, 3, 4\}$, then $c(v_1), c(v_2), c(v_3), c(v_4) \in \{1, 2, 3, 4\}$ as well.
    Since $v_1,v_2,v_3,v_4$ form a clique, one of them has color $4$, forcing its backbone neighbor $v_5$ or $v_7$ to have color $\lambda+4$. The edge $v_6v_8$ places $v_8$ in the upper palette, while its non-backbone edges to $v_5,v_7$ exclude $\lambda+4$. Thus $c(v_8)\in\{\lambda+1,\lambda+2,\lambda+3\}$, and the backbone edge $v_8v_9$ gives $c(v_9)\in\{1,2,3\}$.
\end{proof}

By replacing colors $i$ with $\lambda + 5 - i$ in $c$ we immediately obtain:
\begin{corollary}
    \label{col:n4-gadget}
     Let $c$ be a $\lambda$-backbone $(\lambda + 4)$-coloring of the $N_4$-gadget such that $c(v_6) \in \{\lambda + 1, \lambda + 2, \lambda + 3, \lambda + 4\}$. Then, $c(v_9) \in \{\lambda + 2, \lambda + 3, \lambda + 4\}$.
\end{corollary}

This gadget, in turn, is useful for defining the \emph{$F_{\lambda + 4}$-gadget}, also shown in \Cref{fig:n4-f8-gadget}, and proving its important property:
\begin{lemma}
    \label{lem:f8-gadget}
    Let $c$ be a $\lambda$-backbone $(\lambda + 4)$-coloring of the $F_{\lambda + 4}$-gadget such that $c(u_4) \in \{1, 2, 3, 4\}$ and $c(v_9), c(v'_9) \in \{1, 2, 3\}$. Then, $c(u_3) = \lambda + 4$.
\end{lemma}

\begin{proof}
    Clearly, $c(u_1), c(u_2) \in \{1, 2, 3, 4\}$.
    Since $u_1$, $u_2$, $v_9$, and $v'_9$ form a clique colored with a set $\{1, 2, 3, 4\}$, one of the vertices $u_1$, $u_2$ has to be colored with $4$ in $c$, and it immediately follows that $c(u_3) = \lambda + 4$.
\end{proof}

\begin{corollary}
    \label{col:f8-gadget}
    Let $c$ be a $\lambda$-backbone $(\lambda + 4)$-coloring of the $F_{\lambda + 4}$-gadget such that $c(u_4) \in \{\lambda + 1, \lambda + 2, \lambda + 3, \lambda + 4\}$ and $c(v_9), c(v'_9) \in \{\lambda + 2, \lambda + 3, \lambda + 4\}$. Then, $c(u_3) = 1$.
\end{corollary}

Next, we proceed with two further gadgets, shown below in \Cref{fig:f7-f6-gadget}, $F_{\lambda + 3}$ and $F_{\lambda + 2}$. As their names suggest, they are supposed to force colors $\lambda + 3$ and $\lambda + 2$ on some given vertices, respectively.

\begin{figure}[H]
\centering
\begin{tikzpicture}[scale=1]
\foreach \val [count=\i] in {-2, 0, 2} {
    \ifthenelse{\i = 2}{\def\y{\lambda + 4}}{\def\y{1}}
    \node[draw, circle, fill=black, inner sep=2pt, label={[label distance=1mm]90:$w_{\i}$}] (w\i) at (\val, 0) {};
    \node[draw, circle, fill=black, inner sep=2pt, label={-180:$u_{3,\i}$}] (u\i) at (\val, -1) {};
    \draw (\val - 0.6, -1.8) rectangle (\val + 0.6, -3);
    \draw (w\i) -- (u\i);
    \draw[line width=3pt] (u\i) -- (\val, -1.8);
    \draw[line width=3pt] (\val - 0.2, -1.8) -- (u\i) -- (\val + 0.2, -1.8);
    \node at (\val, -2.4) {$F_{\y}$};
}
\draw[line width=3pt] (w1) -- (w2) -- (w3);
\draw (w1) to[out=10, in=170] (w3);
\end{tikzpicture}
\quad
\begin{tikzpicture}[scale=1]
\node[draw, circle, fill=black, inner sep=2pt, label={above:$z_1$}] (z1) at (-1,0) {};
\node[draw, circle, fill=black, inner sep=2pt, label={above:$z_2$}] (z2) at (1,0) {};
\foreach \val [count=\i] in {-2, 0, 2} {
    \ifthenelse{\i = 3}{\def\x{0}}{\def\x{180}}
    \ifthenelse{\i = 1}{\def\y{1}}{\ifthenelse{\i = 2}{\def\y{\lambda + 3}}{\def\y{\lambda + 4}}}
    \ifthenelse{\i = 1}{\def\z{$u_{3,1}$}}{\ifthenelse{\i = 2}{\def\z{$w_2$}}{\def\z{$u_{3,2}$}}}
    \node[draw, circle, fill=black, inner sep=2pt, label={[label distance=1mm]\x:\z}] (u\i) at (\val, -1) {};
    \draw (\val - 0.6, -1.8) rectangle (\val + 0.6, -3);
    \ifthenelse{\i = 2}{\draw (u\i) -- (\val, -1.8)}{\draw[line width=3pt] (u\i) -- (\val, -1.8)};
    \draw[line width=3pt] (\val - 0.2, -1.8) -- (u\i) -- (\val + 0.2, -1.8);
    \node at (\val, -2.4) {$F_{\y}$};
}
\draw[line width=3pt] (z1) -- (z2);
\draw (u1) -- (z1);
\draw (u2) -- (z2) -- (u3);
\end{tikzpicture}
\caption{$F_{\lambda + 3}$-gadget (left) and $F_{\lambda + 2}$-gadget (right) with their backbone edges in bold.}
\label{fig:f7-f6-gadget}
\end{figure}
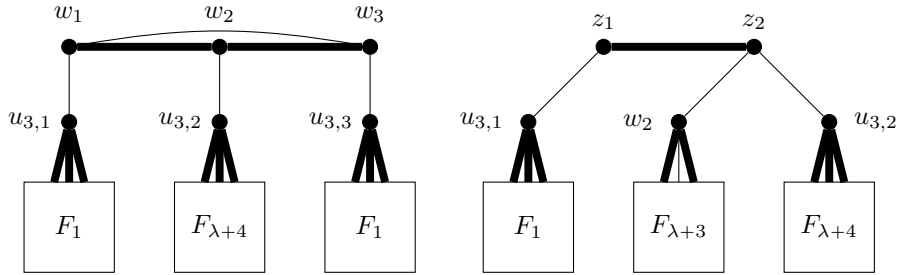

Note that the $F_{\lambda + 3}$-gadget uses not only $F_{\lambda + 4}$-subgadgets, but also so-called $F_1$-subgadgets that will be forcing color $1$ on a given vertex.
These are exactly $F_{\lambda + 4}$-gadgets; the only difference is that in the proofs we will construct $F_{\lambda + 4}$-gadgets with a condition $c(u_4) \in \{1, 2, 3, 4\}$, whereas for $F_1$-gadgets it will be rather $c(u_4) \in \{\lambda + 1, \lambda + 2, \lambda + 3, \lambda + 4\}$ -- thus $c(u_3) = \lambda + 4$ in the $F_{\lambda + 4}$-gadget and $c(u_3) = 1$ in the $F_1$-gadget, respectively. They, in turn, require $N_{\lambda + 1}$-subgadgets instead of $N_4$-subgadgets, identical to $N_4$-gadgets, again with the condition $c(v_6)  \in \{1, 2, 3, 4\}$ switched to $c(v_6) \in \{\lambda + 1, \lambda + 2, \lambda + 3, \lambda + 4\}$ -- thus enforcing $c(v_9) \in \{\lambda + 2, \lambda + 3, \lambda + 4\}$ instead of $c(v_9) \in \{1, 2, 3\}$ (see \Cref{col:n4-gadget}).

\begin{lemma}
    \label{lem:f7-gadget}
    Let $c$ be a $\lambda$-backbone $(\lambda + 4)$-coloring of the $F_{\lambda + 3}$-gadget such that $c(w_1) \in \{1, 2, 3, 4\}$, $c(u_{3,1}) = c(u_{3,3}) = 1$ and $c(u_{3,2}) = \lambda + 4$. Then, $c(w_2) = \lambda + 3$.
\end{lemma}

\begin{proof}
    The backbone path places $w_2$ in the upper palette and $w_1,w_3$ in the lower palette. Their neighbors colored $1$ exclude that color, and $w_1w_3\in E(G)$ makes their colors distinct. Hence $\max\{c(w_1),c(w_3)\}\ge3$, forcing $c(w_2)\ge\lambda+3$. Its neighbor $u_{3,2}$ excludes $\lambda+4$, so $c(w_2)=\lambda+3$. The assignment $(c(w_1),c(w_2),c(w_3))=(2,\lambda+3,3)$ satisfies all edges incident with these vertices under the stated assumptions.
\end{proof}

\begin{lemma}
    \label{lem:f6-gadget}
    Let $c$ be a $\lambda$-backbone $(\lambda + 4)$-coloring of the $F_{\lambda + 2}$-gadget such that $c(z_1) \in \{1, 2, 3, 4\}$, $c(u_{3,1}) = 1$, $c(u_{3,2}) = \lambda + 4$ and $c(w_2) = \lambda + 3$. Then, $c(z_2) = \lambda + 2$.
\end{lemma}

\begin{proof}
    By the backbone edge condition it follows that $c(z_2) \in \{\lambda + 1, \lambda + 2, \lambda + 3, \lambda + 4\}$.
    The colors of $u_{3,2}$ and $w_2$ ensure that $c(z_2) \in \{\lambda + 1, \lambda + 2\}$.
    However, if it were that $c(z_2) = \lambda + 1$, then $c(z_1) = 1$ -- which would contradict $c(u_{3,1}) = 1$.
\end{proof}

Overall, observe that by combining all these lemmas and corollaries above, it is sufficient to ensure that the particular vertices have colors either in $\{1, 2, 3, 4\}$ or in $\{\lambda + 1, \lambda + 2, \lambda + 3, \lambda + 4\}$ ($v_6$ from respective $N_4$-gadgets, $u_4$ from $F_{\lambda + 4}$-gadgets, $w_1$ from $F_{\lambda + 3}$-gadgets, and $z_1$ from $F_{\lambda + 2}$-gadgets) to ensure the colors for particular vertices in $N_4$, $F_{\lambda + 4}$, $F_{\lambda + 3}$, and $F_{\lambda + 2}$-gadgets.
In particular, we distinguish two types of vertices:
\begin{enumerate}
    \item \emph{characteristic vertices}, of degree at most $3$, whose colors are forced once the connection vertices satisfy the hypotheses of the respective lemmas, or their reflected versions: $u_3$ in the $F_{\lambda + 4}$-gadget ($1$ or $\lambda + 4$), $w_2$ in the $F_{\lambda + 3}$-gadget ($2$ or $\lambda + 3$), $z_2$ in the $F_{\lambda + 2}$-gadget ($3$ or $\lambda + 2$),
    \item \emph{connection vertices}, distinct from characteristic vertices, such that their degree is at most $3$: e.g. $v_6$ in the $N_4$-gadget, $u_4$ in the $F_{\lambda + 4}$-gadget, $w_1$ in the $F_{\lambda + 3}$-gadget, $z_1$ in the $F_{\lambda + 2}$-gadget. Note that these are exactly the same vertices that are used in the assumptions of \Cref{lem:n4-gadget,lem:f8-gadget,lem:f7-gadget,lem:f6-gadget}.
\end{enumerate}

Now, we can state the main theorem of this section:
\begin{theorem}
    \label{thm:tree-lambda+4}
    Let $G$ be a graph with $\Delta(G) \le 4$ and $T$ be its spanning tree. Then for every $\lambda \ge 4$ deciding whether $BBC_\lambda(G, T) \le \lambda + 4$ is $\NP$-complete.
\end{theorem}

\begin{proof}
    We proceed by a reduction from the \texttt{3-SAT} problem: given a $3$-CNF instance with $n$ variables $x_1$, \ldots, $x_n$ and $m$ clauses $C_1$, $C_2$, \ldots, $C_m$ we ask whether there exists an assignment satisfying this formula. As in the path reduction, use clauses with three distinct variables and duplicate clauses if necessary to ensure $m \ge 2$.

\begin{figure}[H]
    \centering
    \begin{tikzpicture}[scale=1]
    \foreach \val [count=\i] in {-4, -2, 0, 6} {
        \ifthenelse{\i = 4}{\def\y{2m}}{\def\y{\i}};
        \node[draw, circle, fill=black, inner sep=2pt, label={[label distance=1mm]90:$s^i_{\y}$}] (v\i) at (\val, 0) {};
        \node[draw, circle, fill=black, inner sep=2pt, label={[label distance=1mm]0:$z_{2, \y}$}] (z\i) at (\val, -1) {};
        \draw (\val - 0.6, -1.8) rectangle (\val + 0.6, -3);
        \draw (v\i) -- (z\i) -- (\val, -1.8) [line width=3pt];
        \draw (\val - 0.3, -1.8) -- (z\i) -- (\val + 0.3, -1.8);
        \node at (\val, -2.4) {$F_{\lambda + 2}$};
    }
    \draw (v1) -- (v2) -- (v3);
    \draw[dotted] (v3) -- (v4);
    \draw (v1) to[out=5, in=175] (v4);
    \end{tikzpicture}
    \caption{Variable gadget for $x_i$.}
     \label{fig:variable-gadget}
\end{figure}
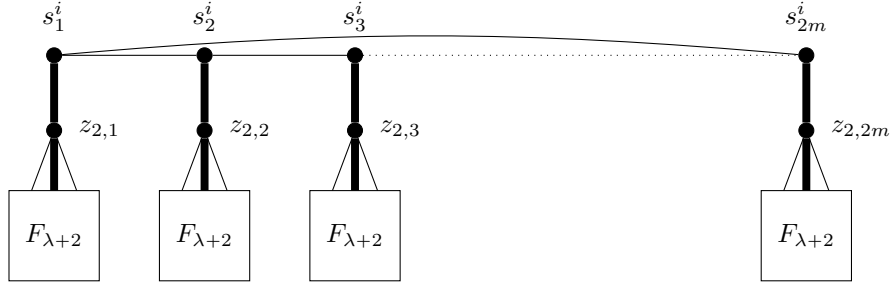

    For every variable $x_i$ we construct a variable gadget (see \Cref{fig:variable-gadget}) by creating $2m$ $F_{\lambda + 2}$-gadgets and a non-backbone cycle on $2m$ vertices $s^i_1$, \ldots, $s^i_{2m}$. Then, we connect $s^i_j$ to a respective characteristic vertex of its $F_{\lambda + 2}$-gadget with a backbone edge, so each $s^i_j$ has degree $3$ at this stage.

\begin{figure}[H]
    \centering
    \begin{tikzpicture}[xscale=0.88, yscale=0.6]
    \foreach \val [count=\i from 1] in {-4, 0, 4} {
        \ifthenelse{\i = 3}{\def\x{}}{\ifthenelse{\i = 2}{\def\x{'}}{\def\x{''}}};
        \node[draw, circle, fill=black, inner sep=2pt, label={90:$u\x_{3,1}$}] (X1\i) at (-3,\val) {};
        \node[draw, circle, fill=black, inner sep=2pt, label={95:$y\x_1$}] (V1\i) at (-1.5,\val) {};
        \node[draw, circle, fill=black, inner sep=2pt, label={90:$y\x_2$}] (V2\i) at (0,\val) {};
        \node[draw, circle, fill=black, inner sep=2pt, label={90:$y\x_3$}] (V3\i) at (1.5,\val) {};
        \node[draw, circle, fill=black, inner sep=2pt, label={90:$y\x_4$}] (V4\i) at (3,\val - 1) {};
        \node[draw, circle, fill=black, inner sep=2pt, label={90:$w\x_2$}] (Y1\i) at (3,\val + 1) {};
        \node[draw, circle, fill=black, inner sep=2pt, label={90:$u\x_{3,2}$}] (Z1\i) at (4.5,\val - 1) {};

        \coordinate (X2\i) at (-4,\val + 0.2);
        \coordinate (X3\i) at (-4,\val - 0.2);
        \coordinate (Y2\i) at (4,\val + 1.4);
        \coordinate (Y3\i) at (4,\val + 0.6);
        \coordinate (Y4\i) at (4,\val + 1);
        \coordinate (Z2\i) at (5.5,\val - 0.8);
        \coordinate (Z3\i) at (5.5,\val - 1.2);
        \draw (-6, \val - 0.6) rectangle (-4, \val + 0.6);
        \node at (-5, \val) {$F_1$};
        \draw (4, \val + 0.4) rectangle (6, \val + 1.6);
        \node at (5, \val + 1) {$F_{\lambda + 3}$};
        \draw (5.5, \val - 1.6) rectangle (7.5, \val - 0.4);
        \node at (6.5, \val - 1) {$F_1$};

        \draw (X2\i) -- (X1\i) -- (X3\i) [line width=3pt];
        \draw (Y2\i) -- (Y1\i) -- (Y3\i) [line width=3pt];
        \draw (Z2\i) -- (Z1\i) -- (Z3\i) [line width=3pt];
        \draw (Y1\i) -- (Y4\i);
        
        \draw (X1\i) -- (V1\i);
        \draw (V2\i) -- (V3\i) -- (Y1\i);
        \draw (Z1\i) -- (V4\i);
        \draw (V1\i) -- (V2\i) [line width=3pt];
        \draw (V3\i) -- (V4\i) [line width=3pt];
    }
    \draw (V11) to[out=80, in=-80] (V12);        
    \draw (V12) to[out=80, in=-80] (V13);        
    \draw (V13) to[out=-75, in=75] (V11);
    \end{tikzpicture}
    \caption{Clause gadget.}
     \label{fig:clause-gadget}
\end{figure}
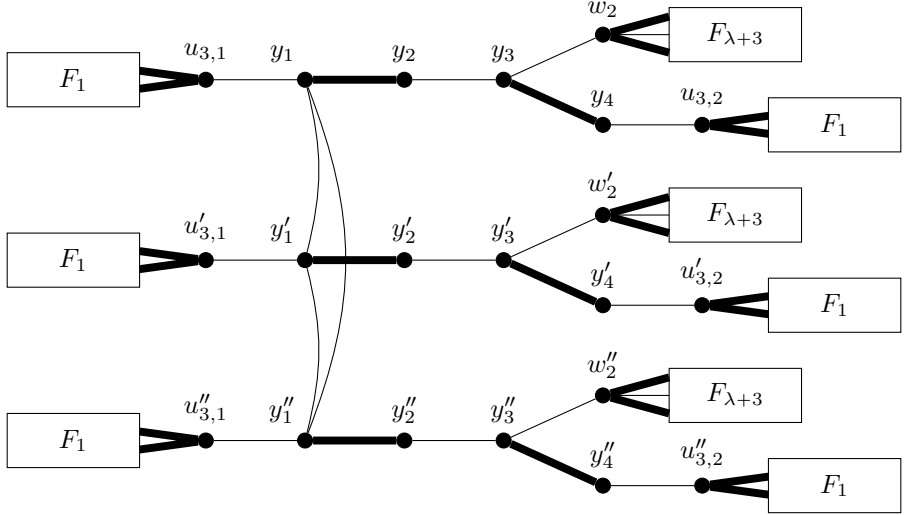

    For every clause we build a gadget presented in \Cref{fig:clause-gadget} and connect vertices $y_4$, $y'_4$, and $y''_4$ with a non-backbone edge to $s^i_{2j}$ (respectively, $s^i_{2j - 1}$) when it was the $j$-th appearance (in the whole formula) of a literal $x_i$ ($\lnot x_i$). Additionally, we define its connection vertices to be $y_2$, $y_3$, $y'_2$, $y'_3$, $y''_2$, and $y''_3$.

    Finally, to ensure that the backbone is a tree, we create an auxiliary binary backbone tree with root $r$ such that its leaves at an even distance from $r$ are exactly all the connection vertices from all $N_4$, $F_{\lambda + 2}$, $F_{\lambda + 3}$, and $F_{\lambda + 4}$-gadgets used in the construction, and its leaves at an odd distance from $r$ are exactly all the connection vertices from all the $F_1$, $N_{\lambda + 1}$, and clause gadgets used in the construction. This way, in any $\lambda$-backbone $(\lambda + 4)$-coloring $c$ of a tree, $c(r) \in \{1, 2, 3, 4\}$ is equivalent to $c(v) \in \{\lambda + 1, \lambda + 2, \lambda + 3, \lambda + 4\}$ for the connection vertices $v$ in $N_{\lambda + 1}$, $F_1$ and clause gadgets, and $c(v) \in \{1, 2, 3, 4\}$ for all other connection vertices $v$.

    To construct this auxiliary tree, start with a binary tree with one leaf for each connection vertex and subdivide each leaf edge once if needed to obtain the prescribed parity of its distance from $r$. Identify each leaf with its assigned connection vertex. Before these identifications the gadget backbone is a forest with exactly one connection vertex in each component; hence the resulting backbone is connected and acyclic. There are $O(nm + m)$ gadget vertices and connection vertices, so the entire construction has polynomial size.

    It is easy to verify that all vertices in variable gadgets, clause gadgets and internal vertices of the auxiliary backbone tree have degree at most $4$, since edges between clause and variable gadgets increase the respective degrees of $s^i_j$ and $y_4$ vertices by $1$.
    Similarly, in all other gadgets we only added at most one edge to connection vertices and characteristic vertices, but since they had degree at most $3$ and all other vertices had degree at most $4$, we conclude that our graph indeed has maximum degree at most $4$.

    Now suppose that there exists an assignment $\phi$ satisfying the formula. Then, we construct the coloring $c$ as follows: if $\phi(x_i) = 1$, then we set $c(s^i_{2j}) = 1$, $c(s^i_{2j - 1}) = 2$, otherwise we set $c(s^i_{2j}) = 2$, $c(s^i_{2j - 1}) = 1$ for all $j = 1, \ldots, m$.
    Since all clauses are satisfied, each clause gadget has at least one adjacent $s$ vertex with color $1$, and two others with colors from $\{1, 2\}$.
    Thus, one of the vertices $y_4$, $y'_4$, $y''_4$ can get color $2$ (without loss of generality we assume $c(y_4) = 2$), while the others get color $3$. This way, we can assign $c(y_3) = \lambda + 2$ and $c(y'_3) = c(y''_3) = \lambda + 4$, and furthermore $c(y_2) = \lambda + 4$, $c(y'_2) = \lambda + 3$, $c(y''_2) = \lambda + 2$, together with $c(y_1) = 4$, $c(y'_1) = 3$, and $c(y''_1) = 2$.

    We complete the coloring by assigning:
    \begin{itemize}
        \item in every $F_{\lambda + 2}$-gadget: $c(z_1) = 2$, $c(z_2) = \lambda + 2$,
        \item in every $F_{\lambda + 3}$-gadget: $c(w_1) = 2$, $c(w_2) = \lambda + 3$, $c(w_3) = 3$,
        \item in every $F_{\lambda + 4}$-gadget: $c(u_1) = c(u_4) = 1$, $c(u_2) = 4$, $c(u_3) = \lambda + 4$,
        \item in every $N_4$-gadget: $c(v_1) = 1$, $c(v_2) = 2$, $c(v_3) = 3$, $c(v_4) = 4$, $c(v_5) = \lambda + 2$, $c(v_6) = 1$, $c(v_7) = \lambda + 4$, $c(v_8) = \lambda + 3$, and
        \begin{itemize}
            \item either $c(v_9) = 2$ when $v_9$ appears also as $v'_9$ in some $F_{\lambda + 4}$-gadget (see \Cref{fig:n4-f8-gadget}),
            \item or $c(v_9) = 3$ when $v_9$ appears also as $v_9$ in some $F_{\lambda + 4}$-gadget.
        \end{itemize}
    \end{itemize}
    The colors in $F_1$ and $N_{\lambda + 1}$-gadgets are exactly the reverse, i.e. we replace color $i$ from $F_{\lambda + 4}$ or $N_4$-gadget coloring with $\lambda + 5 - i$.
    All the edge conditions within gadgets are therefore satisfied and one can check that also the color difference constraints on edges between $(a)$ $F_{\lambda + 4}$-gadget and its $N_4$-subgadgets, $(b)$ $F_{\lambda + 3}$-gadget and its $F_{\lambda + 4}$ and $F_1$-subgadgets, $(c)$ $F_{\lambda + 2}$-gadget and its $F_{\lambda + 3}$, $F_{\lambda + 4}$ and $F_1$-subgadgets are satisfied too.

    Finally, color internal vertices of the auxiliary backbone tree alternately with $1$ and $\lambda+4$, starting with $c(r)=1$. The prescribed leaf parities place each connection vertex in the palette opposite its parent. A parent colored $1$ differs by at least $\lambda$ from every upper-palette color, and a parent colored $\lambda+4$ differs by at least $\lambda$ from every lower-palette color. Thus all auxiliary edges satisfy the backbone condition, completing the coloring.

    Now, let us prove the other direction. Assume that there exists a $\lambda$-backbone $(\lambda + 4)$-coloring $c$ of such a graph $G$ with backbone $T$. Reflecting colors by $a\mapsto\lambda+5-a$ if necessary, assume $c(r) \in \{1, 2, 3, 4\}$.
    This implies that $c(v) \in \{\lambda + 1, \lambda + 2, \lambda + 3, \lambda + 4\}$ for the connection vertices $v$ in $N_{\lambda + 1}$, $F_1$ and clause gadgets, and $c(v) \in \{1, 2, 3, 4\}$ for all connection vertices $v$ in $N_4$, $F_{\lambda + 4}$, $F_{\lambda + 3}$ and $F_{\lambda + 2}$-gadgets. This, in turn, by \Cref{lem:n4-gadget,lem:f8-gadget,lem:f7-gadget,lem:f6-gadget} implies that for every $F_{\lambda + 2}$-gadget with its characteristic vertex $v$ adjacent to $s^i_j$ in any variable gadget we have $c(v) = \lambda + 2$, and thus $c(s^i_j) \in \{1, 2\}$.
    Since for a fixed $i$ the vertices $s^i_j$ form an even cycle, it has to be the case that either $c(s^i_{2j}) = 1$ and $c(s^i_{2j - 1}) = 2$, or $c(s^i_{2j}) = 2$ and $c(s^i_{2j - 1}) = 1$ for all $j = 1, 2, \ldots, m$. Let us set $\phi(x_i) = 1$ in the first case; otherwise we set $\phi(x_i) = 0$.

    Finally, we know that for any clause gadget:
    \begin{description}
        \item[$(1)$] $c(y_2), c(y_3), c(y'_2), c(y'_3), c(y''_2), c(y''_3) \in \{\lambda + 1, \lambda + 2, \lambda + 3, \lambda + 4\}$, since they are the connection vertices in the gadget,
        \item[$(2)$] $c(y_1), c(y_4), c(y'_1), c(y'_4), c(y''_1), c(y''_4) \in \{1, 2, 3, 4\}$,
        \item[$(3)$] by \Cref{col:f8-gadget} we have $c(u_{3,1}) = c(u_{3,2}) = c(u'_{3,1}) = c(u'_{3,2}) = c(u''_{3,1}) = c(u''_{3,2}) = 1$,
        \item[$(4)$] by \Cref{lem:f7-gadget} we have $c(w_2) = c(w'_2) = c(w''_2) = \lambda + 3$.
    \end{description}
    Suppose now that $\phi$ is not a satisfying assignment and pick a clause which is not satisfied. Then, by its definition in the respective clause gadget $y_4$, $y'_4$, $y''_4$ are adjacent to respective $s$ vertices from the variable gadgets colored with $2$ in $c$.
    By conditions $(2)$ and $(3)$ it follows that $c(y_4), c(y'_4), c(y''_4) \in \{3, 4\}$ so $c(y_3), c(y'_3), c(y''_3) \in \{\lambda + 3, \lambda + 4\}$ and by $(4)$ it has to be the case that $c(y_3) = c(y'_3) = c(y''_3) = \lambda + 4$.
    
    However, condition $(3)$ and the triangle on $y_1,y'_1,y''_1$ force these vertices to use all of $2,3,4$. The vertex colored $4$ forces its backbone neighbor among $y_2,y'_2,y''_2$ to have color $\lambda+4$, equal to its adjacent vertex among $y_3,y'_3,y''_3$. This contradicts the corresponding non-backbone edge. Therefore, $\phi$ satisfies the input formula.
\end{proof}

\section{Conclusion}

In this article we proved that $BBC_\lambda(G,H)\le\lambda+5$ for every graph $G$ with maximum degree at most $4$ and every connected spanning bipartite backbone $H$, provided that $\lambda\ge3$. A suitable coloring can be found in linear time. The bound is sharp for $\lambda\ge4$, even for spanning tree backbones.

We also answered the open question of Janczewski and Turowski \cite{janczewski2015computational}: backbone coloring is $\NP$-complete for graphs with maximum degree $4$ and connected bipartite backbones, even when the backbone is a spanning tree, for every fixed $\lambda\ge2$. Together with their subcubic algorithm, this settles the degree boundary for these classes.

The classification in \Cref{tab:summary} is complete for matching and galaxy backbones, both with and without isolated backbone vertices. Within the four backbone classes in the table, only the decision cases $BBC_2(G,P)\le5$ and $BBC_3(G,T)\le7$ remain open.
We conjecture that there might be a simple formula describing the complexity status of the $\lambda$-backbone coloring problem for fixed-degree graphs:
\begin{conjecture}
    Let $\lambda \ge 2$. Let also $G$ be a graph with a maximum degree $\Delta(G) \le 4$. Then for any $l \ge \lambda + 2$ (when the backbone $H$ is a matching or a galaxy) or $l \ge \lambda + 3$ (when the backbone $H$ is a Hamiltonian path or a spanning tree) the problem $BBC_{\lambda}(G, H) \le l$ is either trivial (that is, the answer is \texttt{YES} for all graphs) or $\NP$-complete.
\end{conjecture}
Of course, one can also generalize the above conjecture for graphs with any fixed maximum degree.

The conjecture leaves open which alternative holds in the unsettled cases. We expect $BBC_2(G,P)\le5$ to be $\NP$-complete, as its counterparts for larger $\lambda$ are.
For trees the status does depend on $\lambda$: $BBC_\lambda(G,T)\le\lambda+4$ is trivial for $\lambda=2$ and $\NP$-complete for $\lambda\ge4$, while \Cref{thm:tree-lambda+5} gives a linear-time construction with $\lambda+5$ colors for $\lambda\ge3$. There remains the question whether the upper bound for $\lambda=3$ is best possible; we believe that it can still be improved:
\begin{conjecture}
    There is no pair consisting of a graph $G$ with maximum degree at most $4$ and a spanning tree $T$ such that $BBC_3(G, T) = 8$.
\end{conjecture}

\section*{Acknowledgements}
During the preparation of this work, the authors used OpenAI ChatGPT for verification
of some of the proofs. After using this tool, the authors amended proofs
as needed and take full responsibility for their contents.
Generative AI was also used to find some relevant prior art references and to improve the writing (including spelling
and grammar) for this paper.
In addition, the authors also used generative AI to assist in creating software for verification of correctness for small cases e.g. gadgets.

\section*{Statements and Declarations}

The authors declare that no funds, grants, or other support were received during the preparation of this manuscript. The authors have no relevant financial or non-financial interests to disclose.

\bibliography{degree-4.bib}

@article{lovasz1975three,
  author  = {Lov{\'a}sz, L{\'a}szl{\'o}},
  title   = {Three short proofs in graph theory},
  journal = {Journal of Combinatorial Theory, Series B},
  volume  = {19},
  number  = {3},
  pages   = {269--271},
  year    = {1975},
  doi     = {10.1016/0095-8956(75)90089-1}
}

@misc{baetz2014brooks,
  author        = {Baetz, Bradley and Wood, David},
  title         = {{Brooks'} Vertex-Colouring Theorem in Linear Time},
  year          = {2014},
  eprint        = {1401.8023},
  archivePrefix = {arXiv},
  primaryClass  = {cs.DM},
  url           = {https://arxiv.org/abs/1401.8023},
  note          = {Technical Report CS-AAG-2001-05, 2001; arXiv version, 2014}
}

@inproceedings{schaefer1978complexity,
  author    = {Schaefer, Thomas},
  title     = {The complexity of satisfiability problems},
  booktitle = {Proceedings of the 10th Annual ACM Symposium on Theory of Computing (STOC 1978)},
  editor    = {Lipton, Richard and Burkhard, Walter and Savitch, Walter and Friedman, Emily and Aho, Alfred},
  pages     = {216--226},
  year      = {1978},
  publisher = {Association for Computing Machinery},
  address   = {New York, NY, USA},
  isbn      = {978-1-4503-7437-8},
  doi       = {10.1145/800133.804350}
}

@article{kochol2003colorability,
  author  = {Kochol, Martin and Lozin, Vadim and Randerath, Bert},
  title   = {The {3}-Colorability Problem on Graphs with Maximum Degree Four},
  journal = {SIAM Journal on Computing},
  volume  = {32},
  number  = {5},
  pages   = {1128--1139},
  year    = {2003},
  doi     = {10.1137/S0097539702418759}
}

@article{havet2014circular,
  author  = {Havet, Fr{\'e}d{\'e}ric and King, Andrew and Liedloff, Mathieu and Todinca, Ioan},
  title   = {({Circular}) backbone colouring: Forest backbones in planar graphs},
  journal = {Discrete Applied Mathematics},
  volume  = {169},
  pages   = {119--134},
  year    = {2014},
  doi     = {10.1016/j.dam.2014.01.011}
}

@article{aboulker2017colouring,
  author  = {Aboulker, Pierre and Brettell, Nick and Havet, Fr{\'e}d{\'e}ric and Marx, D{\'a}niel and Trotignon, Nicolas},
  title   = {Coloring graphs with constraints on connectivity},
  journal = {Journal of Graph Theory},
  volume  = {85},
  number  = {4},
  pages   = {814--838},
  year    = {2017},
  doi     = {10.1002/jgt.22109}
}

@article{mivskuf2009backbone,
  author  = {Mi{\v{s}}kuf, Jozef and {\v{S}}krekovski, Riste and Tancer, Martin},
  title   = {Backbone colorings of graphs with bounded degree},
  journal = {Discrete Applied Mathematics},
  volume  = {158},
  number  = {5},
  pages   = {534--542},
  year    = {2010},
  doi     = {10.1016/j.dam.2009.11.015}
}

@article{janczewski2015backbone,
  author  = {Janczewski, Robert and Turowski, Krzysztof},
  title   = {The backbone coloring problem for bipartite backbones},
  journal = {Graphs and Combinatorics},
  volume  = {31},
  number  = {5},
  pages   = {1487--1496},
  year    = {2015},
  doi     = {10.1007/s00373-014-1462-9}
}

@article{janczewski2015computational,
  author  = {Janczewski, Robert and Turowski, Krzysztof},
  title   = {The computational complexity of the backbone coloring problem for bounded-degree graphs with connected backbones},
  journal = {Information Processing Letters},
  volume  = {115},
  number  = {2},
  pages   = {232--236},
  year    = {2015},
  doi     = {10.1016/j.ipl.2014.09.018}
}

@article{araujo2022backbone,
  author  = {Ara{\'u}jo, Camila Sena and Ara{\'u}jo, J{\'u}lio and Silva, Ana Shirley and Cezar, Alexandre},
  title   = {Backbone coloring of graphs with galaxy backbones},
  journal = {Discrete Applied Mathematics},
  volume  = {323},
  pages   = {2--13},
  year    = {2022},
  doi     = {10.1016/j.dam.2022.06.024}
}

@article{michalik2024cliques,
  author  = {Michalik, Krzysztof and Turowski, Krzysztof},
  title   = {On $\lambda$-backbone coloring of cliques with tree backbones in linear time},
  journal = {Journal of Graph Theory},
  volume  = {107},
  number  = {1},
  pages   = {5--28},
  year    = {2024},
  doi     = {10.1002/jgt.23108}
}

@article{araujo2025chordal,
  author  = {Ara{\'u}jo, J{\'u}lio and Nisse, Nicolas and Picasarri-Arrieta, Lucas},
  title   = {Backbone colouring of chordal graphs},
  journal = {Procedia Computer Science},
  volume  = {273},
  pages   = {133--140},
  year    = {2025},
  doi     = {10.1016/j.procs.2025.10.290}
}

@article{broersma2007backbone,
  author  = {Broersma, Hajo and Fomin, Fedor and Golovach, Petr and Woeginger, Gerhard},
  title   = {Backbone colorings for graphs: Tree and path backbones},
  journal = {Journal of Graph Theory},
  volume  = {55},
  number  = {2},
  pages   = {137--152},
  year    = {2007},
  doi     = {10.1002/jgt.20228}
}

@book{murty2008graph,
  author    = {Bondy, John Adrian and Murty, Uppaluri Siva Ramachandra},
  title     = {Graph Theory},
  series    = {Graduate Texts in Mathematics},
  volume    = {244},
  publisher = {Springer},
  address   = {London},
  year      = {2008},
  isbn      = {978-1-84628-969-9},
  doi       = {10.1007/978-1-84628-970-5}
}

@article{broersma2009lambda,
  author  = {Broersma, Hajo and Fujisawa, Jun and Marchal, Bert and Paulusma, Dani{\"e}l and Salman, A. N. M. and Yoshimoto, Kiyoshi},
  title   = {$\lambda$-backbone colorings along pairwise disjoint stars and matchings},
  journal = {Discrete Mathematics},
  volume  = {309},
  number  = {18},
  pages   = {5596--5609},
  year    = {2009},
  doi     = {10.1016/j.disc.2008.04.007}
}

@phdthesis{contributions,
  author  = {Salman, A. N. M.},
  title   = {Contributions to graph theory},
  school  = {University of Twente},
  address = {Enschede, Netherlands},
  year    = {2005},
  isbn    = {978-90-365-2147-5},
  url     = {https://research.utwente.nl/en/publications/contributions-to-graph-theory-2/}
}

@article{fleischner1992solution,
  author  = {Fleischner, Herbert and Stiebitz, Michael},
  title   = {A solution to a colouring problem of {P. Erd{\H{o}}s}},
  journal = {Discrete Mathematics},
  volume  = {101},
  number  = {1--3},
  pages   = {39--48},
  year    = {1992},
  doi     = {10.1016/0012-365X(92)90588-7}
}

@article{fellows1990transversals,
  author  = {Fellows, Michael},
  title   = {Transversals of vertex partitions in graphs},
  journal = {SIAM Journal on Discrete Mathematics},
  volume  = {3},
  number  = {2},
  pages   = {206--215},
  year    = {1990},
  doi     = {10.1137/0403018}
}

@book{garey-johnson,
  author    = {Garey, Michael and Johnson, David},
  title     = {Computers and Intractability: A Guide to the Theory of {NP}-Completeness},
  publisher = {W. H. Freeman and Company},
  address   = {San Francisco, CA, USA},
  year      = {1979},
  isbn      = {978-0-7167-1045-5}
}

@misc{meijer2023,
  author        = {Meijer, Lucas},
  title         = {{3}-Coloring in Time $\mathcal{O}(1.3217^n)$},
  year          = {2023},
  eprint        = {2302.13644},
  archivePrefix = {arXiv},
  primaryClass  = {cs.DS},
  url           = {https://arxiv.org/abs/2302.13644},
  note          = {arXiv preprint}
}

@article{alon1999combinatorial,
  author  = {Alon, Noga},
  title   = {Combinatorial {Nullstellensatz}},
  journal = {Combinatorics, Probability and Computing},
  volume  = {8},
  number  = {1--2},
  pages   = {7--29},
  year    = {1999},
  doi     = {10.1017/S0963548398003411}
}

@article{alon1992colorings,
  author  = {Alon, Noga and Tarsi, Michael},
  title   = {Colorings and orientations of graphs},
  journal = {Combinatorica},
  volume  = {12},
  number  = {2},
  pages   = {125--134},
  year    = {1992},
  doi     = {10.1007/BF01204715}
}

\end{document}